\documentclass[11pt,a4paper]{amsart}
\usepackage{amsfonts,amscd,amssymb,amsmath,amsthm,mathrsfs,xcolor,lscape,ifthen,mathtools,graphicx,fancyhdr,enumerate}
\usepackage[english]{babel}
\usepackage[all,cmtip]{xy}
\usepackage{tikz}
\usepackage{tikz-cd,xr,multicol,microtype}
\usepackage{tkz-euclide}
\usetikzlibrary{matrix}
\usetikzlibrary{calc}
\usepackage{todonotes}
\usepackage{ucs}
\usepackage[utf8]{inputenc}
\usepackage[T1]{fontenc} 
\usepackage[thinlines]{easytable}
\usepackage{color}
\usepackage[colorlinks,linkcolor=blue,anchorcolor=blue,citecolor=blue,backref=page]{hyperref}
\usepackage{hyperref}
\usepackage{phaistos}
\usepackage{microtype}

\newtheorem{thm}{Theorem}[section]
\newtheorem{lem}[thm]{Lemma}
\newtheorem{defn}[thm]{Definition}
\newtheorem{cor}[thm]{Corollary}
\newtheorem{prop}[thm]{Proposition}

\newtheorem{remark}[thm]{Remark}
\newtheorem{note}[thm]{Note}

\newcommand{\mr}[1]{\mathrm{#1}}
\newcommand{\C}{\mathbb{C}}

\newcommand{\R}{\mathbb{R}}

\newcommand{\h}{\mathbb{H}}
\newcommand{\X}{\mathbb{X}}
\newcommand{\Y}{\mathbb{Y}}
\newcommand{\Z}{\mathbb{Z}}
\newcommand{\Q}{\mathbb{Q}}

\newcommand{\G}{\Gamma}

\newcommand{\Cu}{\mathfrak{C}}
\newcommand{\rG}{\mathrm{G}}
\newcommand{\rH}{\mathrm{H}}

\newcommand{\tq}{\widetilde{q}}
\newcommand{\Ii}{\mathrm{i}}

\def\lim{\mathrm{lim}}
\newcommand{\g}{\gamma}

\def\ker{\mathrm{ker}}

\def\cu{{\mathfrak{c}}}

\def\exp{\mathrm{exp}}
\newcommand{\ie}{i.e.\ }
\newcommand{\tmt}[4]{\left({#1\atop #3}{#2\atop #4}\right)}

\newcommand{\eqr}[1]{\mbox{(\ref{eq:#1})}}
\newcommand{\beq}{\begin{equation}}
\newcommand{\eeq}{\end{equation}}

\newcommand\norm[1]{\left\lVert#1\right\rVert}
\newcommand{\Sp}{R(t_{\infty})}

\begin{document}
\date{\today}
\title[vector-valued automorphic forms]{Growth of Fourier Coefficients of vector-valued automorphic forms}
\author{Jitendra Bajpai, Subham Bhakta, Renan Finder}
\noindent\address{Institut f\"ur Geometrie, Technische Universit\"at Dresden, Germany.}
\email{jitendra.bajpai@tu-dresden.de}
\address{Mathematisches Institut, Georg-August-Universit\"at G\"ottingen, Germany.} 
\email{subham.bhakta@mathematik.uni-goettingen.de }
\address{IMPA, Rio de Janeiro, Brasil.}
\email{feliz@impa.br}

\maketitle


\begin{abstract}
In this article, we establish polynomial-growth bound for the sequence of Fourier coefficients associated to even integer weight vector-valued automorphic forms of Fuchsian groups of the first kind. At the end, their $L$-functions and exponential sums have been discussed.
\end{abstract}

\tableofcontents


\section{Introduction}
Let $f:\h \to \C$ be a modular form of level $N$ and weight $k\in \Z$, where $\h:=\{ \tau=x+i y \in \C \ | \ y >0 \},$ denotes the upper half plane. It is known that $f(\tau)$ has a Fourier expansion at any cusp and Fourier coefficients have polynomial-growth. To be precise, it is known that the $n^{\text{th}}$-Fourier coefficient $f_{[n]}\ll n^{k}.$ In this article, whenever we state a bound on the Fourier coefficients, we always mean a bound for absolute value of the same. Also, when $f$ is a cusp form it is known that $f_{[n]}\ll n^{\frac{k}{2}}.$ Both of these bounds are obtained by looking at what happens with the function $F(z)=y^{\sigma}|f(z)|$ in the fundamental domain, and then by comparing $F(z)$ with $F(\gamma z)$ for any $\gamma \in \Gamma_0(N).$ We refer the interested reader to~\cite{Selberg} about the discussion on the sharper bounds of $f_{[n]}$. It is conjectured that $f_{[n]}\ll n^{\frac{k}{2}-\frac{1}{2}+\varepsilon}$ for any $\varepsilon>0.$ When $f$ is a normalized Hecke eigenform, it follows from the multiplicativity of the Fourier coefficients and Deligne's bound $f_{[p^{\alpha}]}\leq (\alpha+1)p^{\alpha\left(\frac{k}{2}-\frac 1 2\right)}$ (see~\cite{Deligne74, BB}) that $f_{[n]}\ll n^{\frac{k}{2}-\frac{1}{2}}d(n),$ where $d(\cdot)$ is the divisor function. This in particular settles down the conjectural bound, since it is known that $d(n)\leq \text{exp}\left(O\left(\frac{\log n}{\log \log n}\right)\right).$ There are some known lower bounds available for $f_{[n]}$ (see~\cite{Murty83}), and these results suggest that Deligne's bound is sharp. 

\par For applications of growth estimates of the Fourier coefficients of modular forms, we refer the interested reader to Sarnak's book~\cite{Sarnak90}. The first application he gives is to show that, if $S^n$ denotes the $n$-dimensional sphere for $n\ge2$ and if $\nu:L^\infty(S^n)\to\mathbb R$ is a linear functional that is invariant under rotations and satisfies $\nu(f)\ge0$ for $f\ge0$, then $\nu$ is a multiple of the integration of linear functional. The second application is to construct highly connected sparse graphs. The third, and the last, is to show that, when $n\to\infty$ in such a way that the number of integer points on the sphere $x^2+y^2+z^2=n$ approaches to infinity, the projections of these points on the unit sphere become equidistributed. For similar results on other quadratic forms than $x^2+y^2+z^2$, the reader may consult~\cite{Iwaniec97} .

\par It is evident that to study growth of Fourier coefficients for modular forms helped in understanding the modular forms  well and hence, if one may allow to say, helped in understanding the geometry of the surfaces as well where they live. For example in the case of modular group, modular forms live at the one punctured Riemann sphere $\mr{PSL}_2(\Z) \backslash \h$ with two special points namely $i$ and the third root of unity $w$. In the case of a Fuchsian group of the first kind $\rG$, the associated fundamental domain $\rG \backslash \h$ will be a Riemann surface with finitely many punctures and special points. Broadly speaking, this article concerns with establishing the growth of Fourier coefficients of vector-valued automorphic forms of Fuchsian groups of the first kind. In order to be explicit about the use of the terms vector-valued modular form (vvmf) and vector-valued automorphic form (vvaf), we will make the following distinction between them: \emph{\textbf{vvaf for a group commensurable with $\mr{PSL}_2(\Z)$ will usually be referred to as vvmf.}} In this sense, we will call our vector-valued functions for Fuchsian groups of the first kind studied in this article \emph{vector valued automorphic forms}.

\par In the same article~\cite{Selberg} mentioned above, Selberg had called for a theory of vector-valued modular forms. Since then numerous attempts have been made and slowly theory has started to emerge. For example they could be an important tool to understand the  modular forms for noncongruence subgroups of the modular group. Observe that every component of $\X(\tau)$ will be a scalar valued modular form for the $\ker(\rho)$ where one could not rule out the possibility of having $\ker(\rho)$ to be a noncongruence subgroup. This could be, in its own, a motivation to study vector-valued modular forms in order to understand scalar valued modular forms for noncongruence subgroups. Later in 1980's, Eichler and Zagier explained in~\cite{EZ} how Jacobi forms and Siegel modular forms can be studied through vector-valued modular forms. For more details on the importance of vector-valued modular forms see the introduction of~\cite{Bajpai2019}. 

\par Roughly speaking, a  vvmf of $\mr{PSL}_2(\Z)$ of weight $k\in 2\Z$ with respect to a representation $\rho: \mr{PSL}_2(\Z)\to \mr{GL}_m(\C)$ is a meromorphic function $\X: \h \to \C^{m}$ which has a certain functional and cuspidal behaviour. For detailed definition and explanation see Section~\ref{sec:vvaf}. Throughout the article $\mr{GL}_m(\C)$ is the group of invertible $m\times m$ matrices with complex entries, and SL$_2(\R)$ (resp. SL$_2(\Z)$) is the group of $2\times 2$ matrices with real (resp. integer) entries and determinant one. Note that the group PSL$_2(\R)= \mr{SL}_2(\R)/\{\pm I \}$, similarly we define the modular group $\mr{SL}_2(\Z)/\{\pm I\}$.

\subsection{Main result}
Let $\mathbb{X}(\tau)$ be a vvmf of weight $k$ for the modular group $\text{PSL}_2(\mathbb{Z})$ associated to a representation $\rho: \text{PSL}_2(\mathbb{Z})\to \text{GL}_m(\C).$ If $\rho(t)$ is diagonalizable, where $t=\tmt{1}{1}{0}{1},$ then each component $\mathbb{X}_i(\tau)$ of $\mathbb{X}(\tau)$ has a convergent $q$-expansion at cusp infinity.  In particular, we can talk about $n^{\text{th}}$-Fourier coefficients $\mathbb{X}_{[i,n]}$'s of $\X_i$'s. We will call such $\rho$ \emph{admissible} which we shall briefly discuss in the next sections. In the same spirit as in the classical scalar-valued case, Knopp and Mason~\cite{KM} showed that all of these $\X_{[i,n]}\ll n^{k+2\alpha},$ and the bound is of order $n^{\frac{k}{2}+\alpha}$ when $\mathbb{X}(\tau)$ is a cusp form, where $\alpha$ is a constant such that $\norm{\rho(\gamma)}\ll \norm{\gamma}^{\alpha}.$ We always denote norm $\norm{\cdot}$ of a matrix in $\mathrm{GL}_n(\mathbb{R})$ as the usual Euclidean norm in $\mathbb{R}^{n^2}.$ 

For general representation of Fuchsian group of the first kind, referred to as \emph{logarithmic}, the associated vvaf $\X(\tau)$ is a linear combination of certain Fourier expansions, where the coefficients are polynomials in $\tau$, see~\cite{Knopp3} and Definition~\ref{def:logvvaf}. Knopp and Mason~\cite{KM2012} have  also studied the growth of the Fourier coefficients for vector-valued modular forms of the modular group with respect to the logarithmic representations.

The main result of this article is a generalization of the  estimates established by Knopp and Mason in~\cite{KM, KM2012}, for the Fourier coefficients of vector-valued automorphic forms. Let us now state the result.

\begin{thm}\label{maintheorem}
Let $\rG$ be a Fuchsian group of the first kind and $\rho:\rG\to \mr{GL}_m(\C)$ be a representation such that all the eigenvalues of the image of each parabolic element are unitary.\footnote{In certain cases we do not need any restriction on the representation, which will be discussed later in Section~\ref{sec:vvaf}.} Let $\cu$ be any cusp of $\rG$. Then there exists a constant $\alpha$, depending on $\rG$, with the following properties. 
\begin{enumerate}
\item[(a)] If $\X$ is a holomorphic vector-valued automorphic form of even integer weight $k$ with respect to $\rho$, then  the sequence of Fourier coefficients of $\X$ at the cusp $\cu$ is $O(n^{k+2\alpha})$. 
\item[(b)]  If $\X$ is a vector-valued cusp form, the sequence of Fourier coefficients is $O(n^{k/2+\alpha})$.
\item[(c)] Moreover, if $k+2\alpha<0$, then $\X\equiv0$.\footnote{We shall later see that the constants are different for admissible and logarithmic case. Here we are considering maximum of them.} 
\end{enumerate}
\end{thm}

It will follow from the proof of Theorem~\ref{maintheorem} that, for unitary representations, $\alpha$ may be taken to $0$. Here by unitary eigenvalue we mean that the eigenvalue has modulus one, and by unitary representation we mean that each element in the image of $\rho$ are unitary matrix. In particular when $\rho$ is 1-dimensional, we have recovered the classical bound for the scalar valued case. The proof is basically divided into two cases: admissible vvaf and logarithmic vvaf.  For their definitions and details see Section~\ref{sec:vvaf}. We study both cases based on a very classical approach, by first looking at what happens to $\norm{\X(z)}$ in a suitable fundamental domain and then to know what happens for arbitrary $\tau$ in $\h$, we write $\tau=\g z$ for $z$ in the fundamental domain and compare $\norm{\X(\gamma z)}$ with $\|\X(z)\|$ for any $\gamma \in \rG.$ In this process we shall show in Lemma~\ref{lem:polygrowth} that the polynomial-growth of $\rho$ based on the structure theorem for elements in the Fuchsian groups, first given by Eichler~\cite[Satz 1]{E}, and later generalized by Beardon~\cite[Theorem 2]{Beardon1975}.

In Section~\ref{sec:rmks}, we will discuss a sufficient criteria for a representation to have polynomial-growth. It turns out that any element of $\rG$ has a sufficiently nice enough decomposition, where there are only finitely many distinct non-parabolic elements. Roughly speaking, this is the reason why polynomial-growth of $\rho$ depends only on the parabolic elements. We record this criteria in Proposition~\ref{prop:criteria}. Furthermore, we shall also see what happens to the meromorphic functions on the upper-half plane, which have functional property with respect to a representation with polynomial-growth. More precisely, we prove
\begin{thm}\label{thm:converse} Let $\rG$ be a Fuchsian group of the first kind, $\X:\mathbb{H}\to \mathbb{C}$ be a non-zero meromporphic function and $\rho:\rG\to \mr{GL}_m(\C)$ be a representation. Suppose that $ \X(\g \tau) = (c\tau+d)^k\rho(\g) \X(\tau), \forall \g \in \rG,\tau \in \mathbb{H},$ where $\g=\tmt{a}{b}{c}{d}.$ Then, we have the following
\begin{enumerate}
\item[(a)] Suppose that, there exists a constant $\zeta>0$ such that an inequality of the form $\norm{\X(x+iy)}\ll y^{-\zeta}$ holds for all $x+iy\in \mathbb{H}.$ If $\rho$ is irreducible, then we have $$\norm{\rho(\gamma)}\ll \norm{\gamma}^{2\zeta-k},~\forall \gamma \in \rG.$$ 
\item[(b)] More generally if $\norm{\X(x+iy)}\ll \max_{0\leq j\leq m-1}\{|x+iy|^{j}y^{-\zeta}\}$ holds for all $x+iy\in \mathbb{H},$ and $\rho$ is irreducible, then $$\norm{\rho(\gamma)}\ll \max\{\norm{\gamma}^{j+2\zeta-k}\}_{0\leq j\leq m-1},~\forall \gamma \in \rG.$$ 
\item[(c)] If $\rho$ is not necessarily irreducible, then some of the irreducible components of $\rho$ has a similar growth.
\end{enumerate}
\end{thm}
The reader may consider this as a converse to Theorem~\ref{maintheorem}. In particular, this shows that the assumption on representation $\rho$ is necessary and sufficient in Theorem~\ref{maintheorem}.

\subsection{Overview of the article} We will now quickly summarize the content of each section. In Section~\ref{sec:fg}, the theory of Fuchsian groups is briefly outlined. There are many resources available on the subject and nothing original is guaranteed in this section. Rather, it serves to the reader a quick introduction of the subject and provides a notational setup for the article. Section~\ref{structure} will prove to be an important tool to establish our main result. In Section~\ref{sec:vvaf}, we discuss in detail the theory of vector-valued automorphic forms for any Fuchsian group of the first kind. Mostly, the Fourier expansions of the automorphic forms at finite and infinite cusps are explained with details to familiarize the reader with the general aspects of the theory of automorphic forms. Section~\ref{af} and Section~\ref{se:logvvaf}, consist the proof of our main Theorem~\ref{maintheorem}. The proof is divided into two parts: in the first part, in Section~\ref{af}, we study the growth of Fourier coefficients of admissible vvaf; whereas in the second part, that is in Section~\ref{se:logvvaf}, we do this for logarithmic vvaf. In Section~\ref{sec:rmks}, we prove our second main result Theorem~\ref{thm:converse}. More precisely, we establish a sufficient criteria for a representation to have polynomial growth. In Section~\ref{L-funct}, we made our attempt to study the $L$-functions attached to vvaf $\X$ and their analytic continuation. We end this article with Section~\ref{sec:app}, discussing some applications and consequences of Theorem~\ref{maintheorem}. The reader can view the content of last two sections as an immediate consequences of our main theorem combining with the classical analytic tools.

\subsection{Notations}\label{notations} 
Every constant appearing in the draft depends on the Fuchsian group, unless otherwise specified. We write $f \ll g$ for $|f|\leq c|g| $ where $c$ is a constant irrespective of the domains of $f$ and $g$, often $f=O(g)$ is written to denote the same. This constant depend on $\rG$ and the representation associated to it, and often only on the Fuchsian group $\rG$. We also write $f\sim g,$ if their domain is in $\mathbb{R},$ and they are asymptotically equivalent, that is, if $\lim_{x\to \infty}\frac{f(x)}{g(x)}=1.$ In the up-coming sections we shall write $z=u+iv$ as a point in any fundamental domain of $\rG$ and $\tau=x+iy$ a point in $\mathbb{H}.$


\section{Fuchsian groups}\label{sec:fg}
The group of all orientation-preserving isometries of the upper half plane $\h$ for the  Poincar\'e metric $ds^2=\frac{dx^2+dy^2}{y^2}$,  coincides with the group PSL$_2(\R)$. A Fuchsian group is a discrete subgroup $\rG$ of PSL$_2(\R)$ for which $\rG\backslash \h$ is topologically a Riemann surface with finitely many punctures. For a quick exposition on the theory of Fuchsian groups, we refer the reader to Chapter 1 of~\cite{Shimura1}. A group $\rG$ in PSL$_2(\R)$ is called discrete, if $\rG$ is a discrete subgroup of PSL$_2(\R)$ with respect to the induced topology of PSL$_2(\R)$.  More explicitly, to define the discreteness of a subgroup $\rG$ of PSL$_2(\R)$, we mean: \\ 
\emph{given any matrix $A \in \rG$, there is an $\epsilon_{A}>0$ such that all the matrices $B (\neq A)$ in $\rG$ have $\mr{dist}(A, B)>\epsilon_{A}$, where $$\mr{dist}(A, B) = \mr{min}\bigg\{\sum_{i,j}\mid A_{ij}-B_{ij}\mid, \sum_{i,j}\mid A_{ij}+B_{ij}\mid \bigg\}\, .$$} The action of any subgroup of SL$_2(\R)$  on $\h$ is the M\"obius action, defined by  
\begin{equation}\label{mobius}
\left({a \atop c}{b \atop d}\right) \cdot \tau = \frac{a\tau+b}{c\tau+d}.
\end{equation}  
Define $\h^{*}=\h \cup \R \cup \{\infty \}$ to be the extended upper half plane of $\mr{PSL}_2(\R)$ and this action can be extended to $\h^{*}$. For any $\g=\pm\tmt{a}{b}{c}{d} \in \mr{PSL}_2(\R)$, the action of $\g$ on $\infty$ is defined as follows: 
\begin{equation}\label{eq:ginfty}
\g \cdot \infty=\mr{lim}_{\tau \mapsto \infty} \frac{a\tau+b}{c\tau+d}=\frac{a}{c} \in \R \cup \{ \infty\},
\end{equation} 
 and for any $x\in \R$, the action is defined similarly by taking the limit $\tau \mapsto x$ in~\eqr{ginfty}.

The elements of $\mr{PSL}_2(\R)$ can be divided into three classes: elliptic, parabolic and hyperbolic elements. An element $\g \in \mr{PSL}_2(\R)$ is elliptic, parabolic or hyperbolic, if the absolute value of the trace of $\g$ is respectively less than, equal to or greater than 2. A point $\tau \in \h^{*}$ is said to be a fixed point of $\g \in \mr{PSL}_2(\R)$ if $\g \cdot \tau=\tau.$ If $\g=\pm\tmt{a}{b}{c}{d}$ is a parabolic element then its fixed point $\tau= \frac{a \mp 1}{c}$ when $a+d=\pm 2$ and $c\neq 0$, in addition $\tau=\infty$ when $c=0$. 
\begin{note}\label{Ac-matrix}\rm $\mr{PSL}_2(\R)$ acts on $\R \cup \{ \infty \}$. Note that in $\h^{*}$ there is only one notion of $\infty$ usually denoted by $i \infty$ but for notational convenience it will be written $\infty$. Following~\eqr{ginfty}, for any $x \in \R$ it is observed that there exists an element $\g=\pm\tmt{x}{-1}{1}{\ 0}$ such that $\g \cdot \infty=x$ which means PSL$_2(\R)$ acts transitively on $\R \cup \{ \infty \}$. For any $x\in \R$, such $\g$ is denoted by $A_{{x}}$.
\end{note}

\begin{defn}\rm
Let $\rG$ be a subgroup of $\mr{PSL}_2(\R)$. A point $\tau \in \h$ is called an elliptic fixed point of $\rG$ if it is fixed by some nontrivial elliptic element of $\rG$, and $\cu \in \R \cup \{\infty\}$ is called a cusp (respectively hyperbolic fixed point) of $\rG$ if it is fixed by some nontrivial parabolic (respectively hyperbolic) element  of $\rG$. Let $\Cu_{\rG}$ denote the set of all cusps of $\rG$ and we define $\h_{\rG}^{*}=\h \cup \Cu_{\rG}$ to be the extended upper half plane of $\rG$.
\end{defn}
For example: if $\rG=\mr{PSL}_2(\R)$ then $\Cu_{\rG}=\R \cup \{ \infty \}$ and
 if $\rG=\mr{PSL}_2(\Z)$ then $\Cu_{\rG}=\Q \cup \{ \infty \}$ consists of the $\rG$-orbit of cusp $\infty$. For any $\tau \in \h^{*}_{\rG}$, let $\rG_{\tau}=\{ \g \in \rG | \g \cdot \tau =\tau \}$ be the stabilizer subgroup of $\tau$ in G.  For any $\cu \in \Cu_{\rG}$, $\rG_{\cu}$ is an infinite order cyclic subgroup of G.  If $\cu =\infty$ then $\rG_{\infty}$ is generated by $t_{\infty} = \pm\left({1 \atop 0}{h_{{\infty}} \atop 1}\right)=t^{h_{{\infty}}}$ for a unique real number $h_{{\infty}} > 0$ called the cusp width of the cusp $\infty$.  In case of $\cu \neq \infty$, $\rG_{\cu}$ is generated by $t_{\cu}= A_{\cu} t^{h_\cu} A_{\cu}^{-1}$ for some smallest real number $h_{\cu} >0$, called the cusp width of the cusp $\cu$ such that $t_{\cu} \in \rG$ where $A_{\cu}= \pm\left({\cu \atop 1}{-1 \atop \ 0}\right) \in \mr{PSL}_2(\R)$ so that $A_{\cu}(\infty)=\cu$, as defined in the Note~\ref{Ac-matrix}. From now on for convenience $h_{{\infty}}$ will be denoted by $h$. For every $\cu \in \Cu_\rG \backslash \{\infty \}$, the elements of $\rG_\cu$ depend on $\cu$. Since $\cu \in \R \cup \{\infty\}$,  there are two possibilities: $\cu \in \R$ or $\cu=\infty$. Consider $\cu \in \R$ and let $\g$ be any element in $\rG_{\cu}$ then $\g= (t_\cu)^{r}$ for some integer $r$, that is, $\g=  A_{\cu} (t^{h_\cu})^{r} A_{\cu}^{-1}$.

\subsection{Fuchsian groups of the first kind}  
The class of all Fuchsian groups is divided into two categories, namely Fuchsian groups of the first and of the second kind. To distinguish between them, a fundamental domain of Fuchsian groups is defined. The fundamental domain, denoted by $\mr{F}_{\rG}$, exists for any discrete group $\rG$ acting on $\h$ and is defined as follows:

\begin{defn}\label{domain}\rm
Let $\rG$ be any discrete subgroup of $\mr{PSL}_2(\R)$. Then a domain (connected open set) $\mr{F}_{\rG}$ in $\h$ is called the fundamental domain of $\rG$, if \begin{enumerate}
\item no two elements of $\mr{F}_{\rG}$ are equivalent with respect to $\rG$,
\item any point in $\h$ is equivalent to a point in the closure of $\mr{F}_{\rG}$ with respect to $\rG$, that is, any $\rG$-orbit in $\h$ intersects with the closure of $\mr{F}_{\rG}$.
\end{enumerate}
\end{defn}

The hyperbolic area of $\mr{F}_{\rG}$ may be finite or infinite. When $\mr{F}_{\rG}$ has finite area then such $\rG$ is a Fuchsian group of the first kind otherwise of the second kind. For example: $\rG=\langle t=\tmt{1}{1}{0}{1} \rangle$ is the simplest example of a Fuchsian group of the second kind. In this article, we are mainly concerned with Fuchsian groups of the first kind. A Fuchsian group $\rG$ will have several different fundamental domains but it can be observed that their area will always be the same. From $\mr{F}_{\rG}$ a (topological) surface $\Sigma_{\rG}$ is obtained by identifying the closure $\widehat{\mr{F}}_{\rG}$ of $\mr{F}_{\rG}$ using the action of $\rG$ on $\widehat{\mr{F}}_{\rG}$, \ie $\Sigma_{\rG}= \widehat{\mr{F}}_{\rG} / {\sim}$ (equivalently $\Sigma_{\rG}= \rG \backslash \h^{*}_{\rG}$).

 \subsection{Structure of words in Fuchsian groups}\label{structure} 
We say a word in a Fuchsian group is an element of the form $C_1C_2\cdots C_s,$ where each $C_i\in \rG.$ A theorem of Eichler~\cite[Satz 1]{E} asserts that there exists a finite set $\rG_{\text{Eichler}}\subset \rG$ such that any $\gamma$ in $\rG$ equals to a product $C_1C_2\ldots C_L$ for some $C_1,C_2
\ldots,C_L$ so that $L$ is bounded by a linear function of order $\log\|\gamma\|$ and each $C_i$ either belongs to $\rG_{\text{Eichler}}$ or is a power of a parabolic element of $\rG_{\text{Eichler}}$. However, this result will not be sufficient for our purpose because we will need to control the powers of parabolic elements appearing in the Eichler's decomposition. 
\par Along the same lines, Beardon~\cite{Beardon1975} gave a decomposition where each $C_i$ is written as a product of elements from a geometrically chosen set of generators. Following Beardon's notations, the number of such elements coming in the product is denoted by $|C_i|.$ These generators, say $\rG^*,$ are precisely the side pairings of a convex fundamental domain. Let $\mr{D}_{\rG}$ be such a convex fundamental domain of $\rG$. We need to understand these $C_i$’s in more details for the work in Section~\ref{se:logvvaf}. It is known that $\widehat{\mr{D}}_\rG$ has finitely many vertices. We say that two vertex of $\widehat{\mr{D}}_\rG$ are equivalent, if and only if they differ by an element of $\rG,$ and denote $v_1\sim v_2.$ We call a vertex as parabolic vertex, if it is a fixed point of a parabolic element of $\rG.$ It is known that the stabilizer of any parabolic vertex is an infinite cyclic group. 
\begin{lem}\label{lem:beardon} There exists a constant $c$ (possibly depending on $\rG$) and a finite subset $\rG_0$ of $\rG$ such that any $C_i$ with $|C_i|>c$, can be written as a product of a parabolic element with an element of $\rG_0.$ Here the parabolic element is of the form $t_{\cu}^{n},$ for some cusp $\cu\in \rG,$ and integer $n.$
\end{lem}

\begin{proof}  
It follows from Theorem 3 of~\cite{Beardon1975} that there exist a constant $c$ (possibly depending on $\rG$) such that any $C_i$ with $|C_i|>c$ can be written as $A_{n_i+1}\cdots A_{n_{i+1}}$ such that $$\widehat{\mr{D}}_\rG,A_{n_i+1}\widehat{\mr{D}}_\rG,\dots, A_{n_i+1}\cdots A_{n_{i+1}}\widehat{\mr{D}}_\rG$$ share a common parabolic vertex, say $v.$ Therefore, we get a sequences to vertices $\{v_j\}_{1\leq j \leq n_{i+1}-n_i}$ in  $\widehat{\mr{D}}_\rG$ such that \[A_{n_i+1}A_{n_i+2}\cdots A_{n_{i}+j}(v_j)=v,~\forall~1\leq j \leq n_{i+1}-n_i.\] For each pair $(v_1,v_2)$ of equivalent vertices, we fix an element $C_{v_1,v_2}\in \rG$ which takes $v_1$ to $v_2.$ We then have
\[A_{n_i+1}A_{n_i+2}\cdots A_{n_{i+1}}C_{v,v_{n_{i+1}-n_i}}(v)=v.\]
In particular, we can write $A_{n_i+1}A_{n_i+2}\cdots A_{n_{i+1}}C_{v,v_{n_{i+1}-n_i}}=P_v^{k},$ where $P_v$ is the parabolic element in $\rG^{*}$ fixing $v.$ This is because, the stabilizer subgroup (in $\rG$) of any parabolic vertex is a cyclic group. The proof is now complete because $\{C_{(v_1,v_2)\mid v_1\sim v_2}\}$ is a finite set. Moreover, this parabolic element $P_v$ is a power of $t_{\cu}$ for some cusp $\cu,$ because the parabolic vertex $v$ is also a cusp by definition.
\end{proof}

\section{Vector-valued automorphic forms}\label{sec:vvaf}

This section reviews the basics of vector-valued automorphic forms that we need to understand and prove Theorem~\ref{maintheorem} and Theorem~\ref{thm:converse}. Rather recently, theory of vector-valued modular forms for modular group has witnessed a fair amount of their development and interest, see the references mentioned in~\cite{Bajpai2019}. Hence, there are a few resources available which could be used to review the fundamental concepts of vector-valued automorphic forms. However, our treatment to vector-valued automorphic forms in this section closely follow~\cite{BajpaiThesis, Bajpai2019, Gannon1, KM, KM2012}.

Let $j: \mr{PSL}_2(\R) \times \C \to \C$ be the function such that for every $\g=\pm\tmt{a}{b}{c}{d}$ in  $\mr{PSL}_2(\R)$ and $\tau \in \C$,  $j(\g, \tau)=c\tau+d.$ One of the important properties of the  $j$-function, which we will make use of, is \begin{equation}\label{eq:property-j} j(\g_1 \g_2, \tau)=j(\g_1, \g_2 \tau) \cdot j(\g_2, \tau)\end{equation} where $\g_1, \g_2 \in \rG$ and $\tau \in \C$\, are such that $\g_2\tau\neq\infty$.
\begin{defn} \label{stroke_operator} \rm If $\X:\h\to\C^m$ is a vector-valued meromorphic function, $\gamma\in \mr{PSL}_2(\R)$ and $k$ is an even integer, we define a vector-valued meromorphic function $\X|_k\gamma$ on $\h$ by setting $\X|_k\gamma(\tau)=j(\gamma,\tau)^{-k}\X(\gamma\tau).$
\end{defn}

It is easy to check that $\X|_k\gamma_1|_k\gamma_2=\X|_k(\gamma_1\gamma_2)$, so the stroke operator induces a right group action on the space of vector-valued meromorphic functions on $\h$. Moreover, if $T\in \mr{GL}_m(\C)$, then $T(\X|_k\gamma)=(T\X)|_k\gamma$. This plays an important role in our article, as it allows us to relate the behaviors of the automorphic forms when we move from one cusp to another. 

\begin{defn} \rm Let $\X:\h\to\C^m$ be a vector-valued meromorphic function. Then
\begin{itemize} 
\item  We say $\X$ has \emph{moderate growth at $\infty$} when there exist $\nu \in \R$ and $Y>0$ such that $\|\X(\tau)\|< \exp(\nu y)$ when $y>Y$. Recall that we are denoting $y=\mr{Im}\tau.$
\item We say $\X$ has \emph{moderate growth at $\cu \in \R$} with respect to $k\in 2\mathbb Z$ when $\X|_kA_\cu$ has moderate growth at $\infty$.
\end{itemize}
\end{defn}

\begin{remark}\label{growth_f_stroke}\rm If $\X$ has moderate growth at $\cu$ with respect to $k$ and $\gamma\in \mr{PSL}_2(\R)$ sends $\infty$ to $\cu$, then $\X|_k\gamma$ also has moderate growth at $\infty$. This can be shown using the equality $\X|_k\g=\X|_kA_\cu|_kA_\cu^{-1}\g$ and the fact that $A_\cu^{-1}\g$ fixes $\infty$, so it is of the form $\tmt{*}{*}{0}{*}$.
\end{remark}
We are going to define what is a vector-valued automorphic form (vvaf) with respect to a representation $\rho:\rG\to \mr{GL}_m(\C)$. To motivate the general definition, we consider initially the case of admissible $\rho$. 
\begin{defn} \rm
Let $\rho:\rG\to \mr{GL}_m(\C)$ be a representation. We say $\rho$ is an {\it admissible representation } of $\rG$ if  $\rho(\gamma)$ is diagonalizable  for every parabolic element $\gamma\in \rG$. Otherwise, we say $\rho$ is a \emph{logarithmic representation }. 
\end{defn}
\begin{remark}\rm\label{rem:unitary} For the admissible case, moderate growth is actually same as saying that $\X$ is meromorphic at $\infty.$ This condition forces all eigenvalues of $\rho(t_{\infty})$ to be unitary. This is because, suppose $\rho(t_{\infty})$ has one eigenvalue with modulus greater than $1$, say the modulus is $r$. That means $|\X_i(\tau+nh)|=r^n|\X_i(\tau)|,~\forall n\in \mathbb{Z}, \tau\in \h$ and some component $\X_i $ of $\X.$ However for any $n, \tau+n h$ and $\tau$ have the same imaginary parts and this contradicts the moderate growth condition.
\end{remark}

\subsection{Admissible vvaf} \label{admissible_vvaf}
\begin{defn} \rm Let $\rG$ be a Fuchsian group of the first kind, $k$ be an even integer, $\rho:\rG\to \mr{GL}_m(\C)$ be an admissible representation and $\X:\h\to\C^m$ be a vector-valued meromorphic function. Then we say that $\X$ is an \emph{admissible vvaf} of weight $k$ with respect to $\rho$ provided that there are finitely many poles of $\X$ in the closure of a fundamental domain and $\X(\tau)$ satisfies following functional and growth conditions.

\begin{enumerate}
\item $ \X|_k\g = \rho(\g) \X, \quad \forall \g \in \rG,$\label{functional}
\item For any cusp $\cu$ of $\rG$, the function $\X$ has moderate growth at $\cu$ with respect to $k$.
\end{enumerate}
A vvaf is called \emph{holomorphic} if it has no poles in $\h$ and, for any cusp $\cu$ of $\rG$, the function $\X|_kA_\cu$ is bounded in some half-plane (contained in $\h$). It is called a \emph{cusp form} if, for any cusp $\cu$, the function $\X|_kA_\cu(\tau)$ approaches to $0$ as $y\to\infty$.
\end{defn}

\begin{remark}\rm
Usually the representation $\rho$ is referred to as the {\emph{representation }} of the automorphic forms of $\rG$. This article mainly explores the theory of automorphic forms with respect to the higher rank representation s $\rho$, that is for $m > 1$, and calls them vvaf.
\end{remark}

\begin{lem}\label{mero-at-c} Let $\X$ be a meromorphic vvaf of weight $k$ for the Fuchsian group $\rG$ with respect to $\rho$ and $\g\in \mr{PSL}_2(\R)$. Then $\X|_k\g$ is a meromorphic vvaf of weight $k$ for $\g^{-1}\rG\g$ with respect to the representation $\g^{-1}\delta\g\mapsto\rho(\delta)$.
\end{lem}

\begin{proof}
From Definition~\ref{stroke_operator}, the function $\X|_k\gamma$ is meromorphic on $\h$, with poles only at the points $\g$ sends to the poles of $\X$. Hence, if $\mr{F}_\rG$ is a fundamental domain for $\rG$, then $\X|_k\g$ has finitely many poles in $\g^{-1}\widehat{\mr{F}_\rG}=\widehat{\g^{-1}\mr{F}}_{\rG }$, which is a fundamental domain for $\g^{-1}\rG\g$. This is the pole behaviour we should have. We now consider the functional equation. Let $\g^{-1}\delta\g\in \g^{-1} \rG\g$. Then $\X|_k\gamma|_k\g^{-1}\delta\g=\X|_k\delta|_k\gamma=(\rho(\delta)\X)|_k\gamma=\rho(\delta)(\X|_k\g)$, as we desired. Finally, let $\cu$ be a cusp of $\g^{-1}\rG\g$. We show that $\X|_k\g$ has moderate growth at $\cu$ with respect to $k$. By definition, this means $\X|_k\g|_kA_\cu=\X|_k\g A_\cu$ has moderate growth at $\infty$. Since $\g A_\cu$ sends $\infty$ to a cusp of $\rG$, this follows from Remark~\ref{growth_f_stroke}.
\end{proof}
\begin{note}\rm$\;$
If $\X$ is a holomorphic vvaf, so is $\X|_k\g$. As a consequence of growth condition and functional behaviour,  $\X(\tau)$ has an infinite series expansion at any cusp $\cu \in \widehat{\Cu}_{\rG}$. These expansions, which are essentially Laurent series expansions, will be referred to as ``Fourier series expansions''. Often these expansions are referred to as $\tq_{\cu}$-expansions with respect to $\cu \in \Cu_{{\rG}}$, where $\tq_{\cu}=\exp\left(\frac{2\pi i A_{{\cu}}^{-1}\tau}{h_{\cu}}\right)$. In addition, for notational convenience we will always use $\tq$ to denote $\tq_{\infty}$.
\end{note}
\begin{lem}\label{fouriersca} Let $f(\tau)$ be a scalar-valued meromorphic function on $\h$ which has no poles when $y \geq Y$ for some $Y>0$ and obeys $f(\tau+h)=\exp(2\pi i \Lambda) f(\tau)$ for every $\tau \in \h$ for some $\Lambda \in \R.$ Suppose that $f(\tau)$ has moderate growth at $\infty$. Then  \begin{equation}\label{eq:fourier-at-infty} \tq^{^{\ -\Lambda}} f(\tau)=\sum_{n=-M}^{\infty} f_{[n]}\ \tq^{^{\ n}},\end{equation} for some $f_{{[n]}}\in \mathbb{C},M\in \mathbb{Z}$, and this sum converges absolutely in $y > Y$.
\end{lem}
\begin{proof} Since $f(\tau)$ has moderate growth at $\infty$, there is an integer $M$ such that $F(\tau)=\tq^{^{\ M-\Lambda}} f(\tau)$ approaches to $0$ as $y \to \infty$ for $0\leq x \leq h$. Note that $F(\tau+h)=F(\tau)$ therefore $g(\tq)=F(\tau)$ is a well defined and holomorphic function in the punctured disc $0 < |\tq|< \exp(-\frac{2\pi Y}{h})$, about $\tq =0$  and is bounded there (because it approaches to $0$ as $\tq$ goes to $0$). This means that $\tq=0$ is a removable singularity thus defining $g(0)=0$ gives $g(\tq)$ is holomorphic in the disc $|\tq|< \exp(-\frac{2\pi Y}{h})$. This means that $g(\tq)$ has a Taylor expansion in $\tq$ which converges absolutely in that disc.
\end{proof}
For each eigenvalue $\lambda$ of $\rho(t_{\infty}),$ we denote $\mu(\lambda)$ to be the unique real number such that $\lambda=\exp(2\pi i\mu(\lambda))$ and $0\le\mu(\lambda)<1$. 

\begin{prop}\label{fourier}
Let $\rG$ be a Fuchsian group of the first kind with a cusp at $\infty$ and $k$ be an even integer. Let $\X$ be a meromorphic vvaf of weight $k$ with respect to the admissible representation $\rho:\rG\to \mr{GL}_m(\C)$. Let $\rho(t_\infty)=P\mathrm{diag}\left(\lambda_1,\lambda_2,\cdots,\lambda_m\right)P^{-1}$. Then, at the cusp $\infty$,
\begin{equation}\label{fourier_series_vvaf}
    \X(\tau) = P \tq^{\Lambda} P^{-1}\sum_{n=-M}^\infty \X_{[n]} \tq^{n} 
\end{equation}
where $\X_{[n]}\in\C^m$ and $M\in\mathbb Z$. Here $\tq^{\Lambda}$ is denoted to be the diagonal matrix $\mathrm{diag}\left(\tq^{\mu(\lambda_1)}, \tq^{\mu(\lambda_2)},\cdots, \tq^{\mu(\lambda_m)}\right)$ 
\end{prop}

\begin{proof}
We have $P^{-1}\X(\tau+h)=\mathrm{diag}\left(\lambda_1,\lambda_2,\cdots,\lambda_m\right)P^{-1}\X(\tau)$. Hence, each component of the function $\tau\mapsto P^{-1}\X(\tau)$ satisfies the hypotheses of Lemma~\ref{fouriersca}. Applying this, we get 
\begin{equation}\label{eqn:expansion adm}
P^{-1}\X(\tau)=\tq^{\Lambda} \sum_{n=-M}^\infty v_n \tq^{n} 
\end{equation}
for some vector-valued sequence $v_n$. Now we multiply both sides of the last equation by $P$ and define $\X_{[n]}=Pv_n$.
\end{proof}

\begin{remark}\rm
If $v$ is an eigenvector of $\rho(t_{\infty})$ with eigenvalue $\lambda$, then $v$ is an eigenvector of $P\tq^{\Lambda}P^{-1}$  
with eigenvalue $ \tq^{\mu(\lambda)}$. 
 Since the eigenvectors of $\rho(t_{\infty})$ span $\C^m$, this implies that the Fourier expansion at~(\ref{eqn:expansion adm}) does not depend on the choice of the diagonalizing matrix.
\end{remark}

If $\X$ is a holomorphic vvaf then all terms of the sum~\eqref{fourier_series_vvaf} with $n<0$ must vanish. Indeed, if some of these terms did not vanish, the infinite series would grow at least as $\exp\left(2\pi y/h\right)$ as $y\to\infty$, so $\X(\tau)$ would tend to $\infty$ as $y\to\infty$, contradicting our definition of holomorphic automorphic forms.
Similarly, if $\X$ is a cusp form, we may take $\mu(\lambda)$ such that $0<\mu(\lambda)\le1$ for each eigenvalue $\lambda$ of $\rho(t_{\infty})$ (so now $\mu(\lambda)$ might be $1$, but not $0$). Then all the terms with $n\leq0$ vanish. So, for cusp forms, the infinite series in~\eqref{fourier_series_vvaf} is bounded, while the matrix $P\tq^\Lambda P^{-1}$ approaches to $0$ exponentially as $y\to\infty$. Hence $\X(\tau)\to0$ exponentially as $y\to\infty$.

\begin{defn} \rm
The \emph{Fourier expansion of $\X$ at $\infty$} (with respect to the choice of $\mathrm{diag}\left(\mu(\lambda_1),\mu(\lambda_2),\cdots,\mu(\lambda_m)\right)$ is given by~\eqref{fourier_series_vvaf}. The coefficients $\X_{[n]}$ are known as \emph{Fourier coefficients of} $\X$. The \emph{Fourier coefficients of $\X$ at a cusp $\cu\in\R$} are defined to be the Fourier coefficients of $\X|_kA_\cu$. 
\end{defn}

\subsection{Logarithmic vvaf}  In this section, we shall generalize the notion of admissible vvaf and study their growth. Following the usual set up, let $\rG$ be a Fuchsian group of the first kind, $k$ be an even integer, $\rho:\rG\to \mr{GL}_m(\C)$ be a representation and $\X:\h\to\C^m$ be a vector-valued meromorphic function. Suppose that $\X(\tau)$ satisfies the functional behavior $ \X|_k\g = \rho(\g) \X,~\forall \g \in \rG.$ We are interested in the case when $\rho$ is not necessarily admissible and all eigenvalues of $\rho(\gamma)$ are unitary for every parabolic element $\gamma\in \rG$. Such a vvaf $\X$ will be called as \emph{logarithmic vvaf}. We shall be discussing the properties and features of logarithmic vvaf, following~\cite{Knopp3}, required to prove Theorem~\ref{maintheorem}. Let us first consider the space $$W=\text{Span}_{\mathbb{C}}\left\{\X_i(\tau)\mid 0\leq i\leq m-1\right\}.$$ 
Note that $W$ has dimension at most $m$ over $\mathbb{C},$ and it is invariant the action of $t_{\infty}.$ In other words, we can consider $\rho(t_{\infty}):  W\to W$ defined by  $ \X_{i}(\tau) \mapsto \X_{i}(\tau+h).$ With respect to the basis $\left\{\X_i(\tau)\mid 0\leq i\leq m-1\right\}_{0\leq 1\leq m-1}$, or possibly a subset of this if they are linearly dependent, we may assume that $\rho(t_{\infty})$ is in the Jordan canonical form
\[ \begin{psmallmatrix}
   J_{m(\lambda_1),\lambda_1} & & \\
    & J_{m(\lambda_2),\lambda_2} & &\\
    & & \ddots &\\ 
    & & & J_{m(\lambda_k),\lambda_k}\\ 
  \end{psmallmatrix},\]
  where the Jordan block $J_{m_i,\lambda_i}$ is defined to be 
{\small $$\begin{psmallmatrix}
   \lambda_i & & &  &\\
    \lambda_i &  \ddots&  & &\\ 
     & \ddots &  & \ddots &\\
     &   & \lambda_i &  &\lambda_i\\ 
 \end{psmallmatrix},$$ }
which is conjugate to the canonical Jordan block, and $m(\lambda)$ is the multiplicity of the eigenvalue $\lambda.$ We shall denote $\Sp$  to be the set of all eigenvalues of $\rho(t_{\infty}).$

\begin{lem}\label{lem:LC3} Let $\X$ be a holomorphic function on $\mathbb{H}$ such that $\X(\tau+h)=\rho(t_\infty)\X(\tau)$, then for each eigenvalue $\lambda$ of $\rho(t_{\infty})$ there are $\tq$-expansions $h_{\lambda,j}(\tau)=\sum_{n\in \mathbb{Z}} \X_{[\lambda,j,n]}\tq^{n+\mu(\lambda)},~0\leq j\le m(\lambda)-1$ such that 
\begin{equation} \label{fourier_expansion_X_i}
\X(\tau)=\sum_{\lambda \in \Sp}\sum_{j=0}^{m(\lambda)-1}(\log \tq)^{j}h_{\lambda,j}(\tau),
\end{equation}
\end{lem}

\begin{proof} We start by writing
\[\Big\{\X_{i}(\tau)\Big\}_{0\leq i\leq m-1}=\bigsqcup_{\substack{\lambda\in \Sp, \\ 0\leq i\leq m(\lambda)-1}} \Big\{\X_{i, \lambda}(\tau)\Big\},\]
 such that for each eigenvalue $\lambda,$ we have $\rho(t_{\infty})$ is the single block $J_{m(\lambda),\lambda}$ when acting on the space generated by $\{\X_{i,\lambda}\}_{0\leq i\leq m(\lambda)-1}$. 

For each $\lambda$ we can now write,
\[\X_{i,\lambda}(\tau+h)=\lambda(\X_{i,\lambda}(\tau)+\X_{i-1,\lambda}(\tau)),~0\leq i\leq m(\lambda)-1,\]
where we set $\X_{-1,\lambda}=0.$ Define,
\[\tilde{h}_{i,\lambda}(\tau)=\sum_{j=0}^{i}(-1)^{j}\binom{\tau/h+j-1}{j}\X_{i-j,\lambda}(\tau),~0\leq i \leq m(\lambda)-1.\]
Following the argument at page 265 of~\cite{Knopp3}, we see that each $\tilde{h}_{i,\lambda}$ has a convergent $\tq$-expansion of type $\sum_{n\in \mathbb{Z},n+\lambda\geq 0} a_i(n)\tq^{n+\mu(\lambda)}, 0\leq i \leq m(\lambda)-1$ and
\begin{equation}\label{eqn:matrix}
    \X_{i,\lambda}(\tau)=\sum_{j=0}^{i}\binom{\tau}{j}\tilde{h}_{i-j,\lambda}(\tau),~0\leq i\leq m(\lambda)-1.
\end{equation}
Now note that,
\begin{equation*}\label{eqn:span}
\mathrm{Span}_{\mathbb{C}}\left\{\binom{\tau}{j} \mid 0\leq j\leq m(\lambda)-1\right\}=
\mathrm{Span}_{\mathbb{C}} \left\{(\log \tq)^j \mid 0\leq j\leq m(\lambda)-1\right\},
\end{equation*}
because $2\pi i\tau/h=\log \tq.$ It now follows from~(\ref{eqn:matrix}) that there exists a matrix $H_{\lambda}(\tau),$ whose entries are written in terms of $\tilde{h}_{j,\lambda}$'s, such that 
\begin{align*}
    \begin{pmatrix}
           \X_{0,\lambda}(\tau) \\
           \X_{1,\lambda}(\tau) \\
           \vdots \\
           \X_{m(\lambda)-1,\lambda}(\tau)
         \end{pmatrix} &= H_{\lambda}(\tau) A\begin{pmatrix}
           1 \\
           (\log \tq) \\
           \vdots \\
           (\log \tq)^{m(\lambda)-1}
         \end{pmatrix},
  \end{align*}
for some $A\in \text{GL}_m(\mathbb{C}).$ We can therefore write for eigenvalue $\lambda$ and $i\in\{0,1,\cdots,m(\lambda)-1\}$ that
\[\X_{i,\lambda}(\tau)=\sum_{j=0}^{m(\lambda)-1}(\log \tq)^{j}h_{i,j,\lambda}(\tau),\]
where each $h_{i,j,\lambda}(\tau)$ is of form $\sum_{n\geq 0}\X_{[i,j,\lambda,n]}\widetilde{q}^{n+\mu(\lambda)}.$ The result is now proved by taking $h_{\lambda,j}(\tau)=\sum_{0\leq i\leq m(\lambda)-1} h_{i,j,\lambda}(\tau)e_{\lambda, i},$ where $e_{\lambda, i}$ is an element of the canonical basis of $\C^m$.
 \end{proof}
 \begin{note}\rm$\;$
 The lemma above shows that how to get a logarithmic expansion of $\X(\tau)$ at $\infty$ when $\rho(t_{\infty})$ is in the Jordan canonical form. For a general $\rho$, let $P\rho(t_\infty)P^{-1}$ be in the Jordan canonical form, then $\Y=P\X$ is a logarithmic vvaf for $P\rho P^{-1}$. Since $P\rho(t_\infty)P^{-1}$ is in the Jordan canonical form, we have 
$$\Y(\tau)=\sum_{\lambda \in \Sp} \sum_{j=0}^{m(\lambda)-1}(\log \tq)^j\sum_{n=-M}^\infty\tq^{n+\mu(\lambda)}v_{[\lambda,j,n]}$$ 
where $v_{[\lambda,j,n]}$ denotes a vector. Then we need only to multiply this by $P^{-1}$. Since multiplying by $P^{-1}$ will mix the components of $v_{[\lambda,j,n]}$, we have $\tq^{n+\mu(\lambda)}$ for all the eigenvalues $\lambda$.

\par To get such an expansion around other cusp $\cu,$ one needs to get an expansion of $\X|_kA_{\cu}$ around $\infty$, as we did in the admissible case. Having this in our hand, we are now ready to describe the vector-valued automorphic forms when the corresponding representation is not admissible. 
\end{note}
In Lemma~\ref{lem:LC3} and the remark above, we have seen how to get a logarithmic expansion of $\X(\tau).$ We are interested to study the growth when all the components $\X_i(\tau)$ are holomorphic at the cusps. Equivalently, we now have the following definition, which will be used in the rest of the article.
\begin{defn}\label{def:logvvaf} We say $\X$ is holomorphic at the cusp $\infty$ when it has an expansion of the form
$$\X(\tau)=\sum_{\lambda\in \Sp}\sum_{j=0}^{m(\lambda)-1}(\log \tq)^j\sum_{n=0}^\infty\X_{[\lambda,j,n]} \tq^{n+\mu(\lambda)}$$ 
where  each $\X_{[\lambda,j,n]}$ is a vector. We say $\X$ is holomorphic at the cusp $\cu$ when $\X|_kA_\cu$ is holomorphic at the cusp $\infty$. If $\X$ is holomorphic at all cusps, then we say that $\X$ is a holomorphic vvaf. Moreover, we say that $\X(\tau)$ vanishes at the cusp $\infty$ if it has an expansion of the form
$$\X(\tau)=\sum_{\lambda \in \Sp} \sum_{j=0}^{m(\lambda)-1}(\log\tq)^j\sum_{n=1}^\infty\X_{[\lambda,j,n]} \tq^{n+\mu(\lambda)}.$$ 
In other words, all of the associated $\widetilde{q}$-expansions of $\X(\tau)$ have exponential decay as the imaginary part of $\tau$ goes to $\infty.$ Similarly, we say that $\X(\tau)$ vanishes at the cusp $\cu$ if $\X|_kA_{\cu}$ vanishes at the cusp $\infty,$ and then we say that $\X(\tau)$ is a cusp form if $\X(\tau)$ vanishes at all cusps of $\rG.$
\end{defn}


\section{Growth for admissible vector-valued automorphic forms}\label{af}
Before proving Theorem~\ref{maintheorem}, we briefly summarize our strategy. As the cusp may be moved to $\infty$ using $A_\cu$, we may assume that $\cu = \infty$. Applying a theorem of Eichler, we shall show the existence of $\alpha$ such that $$\|\rho(\gamma)\|=\norm{\rho\tmt{a}{b}{c}{d}}\ll(c^2+d^2)^\alpha$$ when $\rho$ is admissible. We choose a bounded fundamental domain for $\rG$. An arbitrary $\tau\in\h$ is picked, with the aim of bounding $\|\X(\tau)\|$. Then we take $\gamma\in \rG$ and $z$ in the fundamental domain such that $\tau = \gamma z$. The vectors $\X(\tau)$ and $\X(z)$ are related via the functional equation, in which there appears $\rho(\gamma)$, whose norm  will be estimated. Using the Fourier expansion at any cusp of $\rG$, one can estimate $\X(z)$ as $z$ approaches to the cusp within the fundamental domain. It turns out that, in the case of cusp forms, $\X(z)$ decays sufficiently fast to ensure that $y^{k/2+\alpha}\X(\tau)$ is bounded in $\h$, where $y=\mr{Im}\tau$. Then from the integral representation of the Fourier coefficients we conclude the proof.

In the case of holomorphic vector-valued automorphic forms, $\X(z)$ may grow like $(\mr{Im}z)^{-k}$ as $z$ approaches to a cusp, so we get a weaker bound, that is, $y^{k+2\alpha}\X(\tau)$ is bounded when $y$ is sufficiently small. In addition, $j(\gamma, z)$ appears in the computations, so we need Corollary~\ref{lowerboundjgammaz} to complete the proof. Also, because of this corollary which only holds for cusps inequivalent to $\infty$, the case of $\infty$ must be treated separately.

Our proof of the main theorem depends on the following result. 

\begin{lem}\label{lemma1}\cite[Lemma 6]{K}
For any $\gamma=\tmt{a}{b}{c}{d}$ in $\rG$, there exists an integer $n$ such that the real numbers $\tilde a$ and $\tilde b$ defined by $\gamma=\tmt{1}{h}{0}{1}^n\tmt{\tilde a}{\tilde b}{c}{d}$ satisfy $\tilde a^2+\tilde b^2\le k_{1}(c^2+d^2)$, where $k_{1}$ is a constant depending only on $\rG$.
\end{lem}
Consequently, we have a polynomial-growth of $\rho$ as follows.
\begin{lem}\label{lem:alpha} For any $\gamma=\tmt{a}{b}{c}{d}$ in $\rG$, we have $\norm{\rho(\gamma)} \ll (c^2+d^2)^{\alpha},$
where $\alpha=O\left(\log (M_\rG)\right)$ and $M_\rG=\max\left\{\norm{\rho(\gamma)}\right\}_{\gamma \in \rG_{\mathrm{Eichler}}}.$ 
\end{lem}

\begin{proof}
From the previous lemma we can write $\gamma=\tmt{1}{h}{0}{1}^n\tmt{\tilde a}{\tilde b}{c}{d},$
such that $\tilde a^2+\tilde b^2 \leq k_1 (c^2+d^2),$ where $k_1$ is some constant depending on $\rG.$ The admissibility of $\rho$ implies that the powers of $\norm{\rho\left(\tmt{1}{h}{0}{1}\right)}$ are uniformly bounded, and in particular, $\norm{\rho(\gamma)}\ll \norm{\rho\left(\tmt{\tilde a}{\tilde b}{c}{d}\right)}.$ Now applying the result of Eichler on $\tmt{\tilde a}{\tilde b}{c}{d},$ we get $\norm{\rho(\gamma)}\ll M_G^{L},$
where $L\leq C_1\log (\tilde a^2+\tilde b^2+c^2+d^2)+C_2\leq C_1\log \left((k_{1}+1)(c^2+d^2)\right)+C_2,$ and $C_1,C_2$ are some constants depending on $\rG.$ In particular, we then have
\[\norm{\rho(\gamma)}\leq M_\rG^{C_2}\times \left((k_{1}+1)(c^2+d^2)\right)^{C_1 \log \left(M_\rG\right)}.\]
In particular, we can now take $\alpha=C_1\log (M_\rG)$ to complete the proof.
\end{proof}

Now, we establish the bound for the Fourier coefficients of cusp forms as stated in Theorem~\ref{maintheorem}.

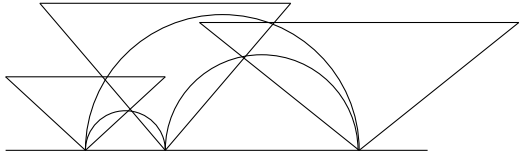
\begin{figure}
\centering
\begin{tikzpicture}[scale=1.5]
\draw (-1.7,0)--(2,0); 
\draw (-1.0, 0) arc [radius=1.2, start angle=180, end angle=0];
\draw (-1.0, 0) arc [radius=0.35, start angle=180, end angle=0]; 
\draw (-0.3, 0)  arc [radius=0.845, start angle=180, end angle=0];
\draw (-1.0, 0)--(-1.7, 0.65);
\draw (-1.0, 0)--(-0.3, 0.65);
\draw (-1.7, 0.65)--(-0.3, 0.65);
\draw (-0.3, 0)--(-1.4, 1.3);
\draw (-0.3, 0)--(0.8, 1.3);
\draw (-1.4, 1.3)--(0.8, 1.3);
\draw (1.4, 0)--(0, 1.13);
\draw (1.4, 0)--(2.8, 1.13);
\draw (0, 1.13)--(2.8, 1.13);
\end{tikzpicture}
\caption{Fundamental domain of a Fuchsian group covered by triangles $\mathcal S(\cu,v_0,K)$.}\label{ttypes}
\end{figure}

\subsection{Proof of part (b) of Theorem~\ref{maintheorem} for vector-valued cusp forms}
With the polynomial-growth of $\rho$ obtained from the previous lemma (together with the functional equation), we want to relate $\X(\gamma z)$ to $\X(z)$. Denoting $\tau = \gamma z$, $\|\X(\tau)\|=\|(cz+d)^k\rho(\g)\X(z)\| \ll|cz+d|^k(c^2+d^2)^\alpha\|\X(z)\|$. Let $z=u+iv$. Using the elementary inequality $c^2+d^2\le|cz+d|^2(1+4|z|^2)/v^2$, proven by Knopp in~\cite[lemma 4]{K}, one obtains
\begin{equation}\label{bound_X}
\|\X(\tau)\|\ll|cz+d|^{k+2\alpha}(1+4|z|^2)^\alpha v^{-2\alpha}\|\X(z)\|.
\end{equation}
Applying the identity $y=\mr{Im}\gamma z = \frac{v}{|cz+d|^2}$, we get 
\begin{equation}\label{bound_product}
y^{k/2+\alpha}\|\X(\tau)\|\ll(1+4|z|^2)^\alpha v^{k/2-\alpha}\|\X(z)\|.
\end{equation}
At this point, it is convenient to restrict $z$ to a fundamental domain of $\rG$ which does not depend on $\tau$. Since $\rG$ is a Fuchsian group of the first kind, $\rG\backslash \h_{\rG}^{*}$ is compact and there are only finitely many equivalence classes of cusps, and a bounded fundamental domain that may be partitioned into a finite set of pieces, see Figure 1. More precisely, the constants $K,v_0,$ and a finite set of cusps $\cu_{\rG}$ such that each such piece is contained in a triangle of the type $\{z\in\h:v<v_0\text{ and }|u-\cu|\leq Kv\}$, which we denote by $\mathcal S(\cu,v_0,K)$, where $\cu \in \Cu_{\rG}.$  Then, for $z$ in this fundamental domain,
\begin{equation}\label{another_bound}
y^{k/2+\alpha}\|\X(\tau)\|\ll v^{k/2-\alpha}\|\X(z)\|.
\end{equation}

Since $\X$ is a cusp form, if $\cu$ is any cusp, $\X|_kA_\cu$ decays exponentially as the imaginary part of its argument goes to infinity. We shall show this implies that, for any real number $\beta$, $\|\X(z)\|\ll v^\beta$ in $\mathcal S(\cu,v_0,K)$. To do so, let $\mathbb{Y}_\cu = \X|_kA_\cu$. Then $$\X(z) = \mathbb{Y}_\cu|_kA_\cu^{-1}(z)=(\cu-z)^{-k}\mathbb{Y}_\cu\left(\frac{1}{\cu-z}\right).$$ Since $|\cu-z|$ is comparable to $v$, $$\|\X(z)\|\ll v^{-k}\|\mathbb{Y}_\cu\left(\frac{1}{\cu-z}\right)\|.$$ Since $\mathbb{Y}_\cu$ decays more rapidly than the $(k+\beta)^{\text{th}}$-power of the imaginary part of its argument, we get $$\|\X(z)\|\ll  v^{-k}\left(\mr{Im}\frac{1}{\cu-z}\right)^{-k-\beta}.$$ Note that $$\mr{Im}\frac{1}{\cu-z}=\frac{v}{|\cu-z|^2}=\frac{v}{(\cu-u)^2+v^2},$$ whence, from the definition of $\mathcal S(\cu,v_0,K)$, $$\frac{1}{v}\ge \mr{Im}\frac{1}{\cu-z}\ge\frac{v}{K^2v^2+v^2}\gg\frac{1}{v}.$$ Therefore, $$\X(z)\ll v^{-k}v^{k+\beta}=v^\beta, ~\forall z \in \mathcal S(\cu,v_0,K).$$ Since the fundamental domain we chose is contained in a finite union of these sets, the bound holds in the fundamental domain as well. Taking $\beta = \alpha-k/2$ and using~\eqref{another_bound}, we see that $y^{k/2+\alpha}\|\X(\tau)\|$ is bounded in $\h$. Now note that the $i^{\text{th}}$-component of the $n^{\text{th}}$-Fourier coefficient is
\[\mathbb{X}_{[i,n]}=\frac{1}{h}\int_{0}^{h}\mathbb{X}_i(x+iy)\tq^{(-n-\mu(\lambda_{i}))} dx.\]
In particular, we then have
\[\mathbb{X}_{[i,n]} \ll y^{-k-2\alpha} e^{2\pi y(n+\mu(\lambda_{i})/h}.\]
Taking $y=\frac{1}{n+\mu(\lambda_{i})}$ we get the desired result.\qed

\subsection{Proof of Theorem~\ref{maintheorem} for  holomorphic vvaf}\label{se:vvhaf}
We now establish the growth for admissible holomorphic vvaf. To complete the proof of the theorem, we need the following lemma.

\begin{lem}\label{lowerboundj}
Let $\cu$ and $\infty$ be the cusps of the Fuchsian group $\rG$. Then either $\cu$ and $\infty$ are equivalent cusps or $\inf_{\gamma\in \rG} |j(\gamma,\cu)|>0$.
\end{lem}
\begin{proof}
We begin with the case $\cu=0$. Note that, if $\gamma=\tmt{a}{b}{c}{d}$, then $j(\gamma,0)=d$. Now assume that $0$ and $\infty$ are not equivalent cusps of $\rG$. Then $d\neq0$ whenever $\tmt{a}{b}{c}{d}\in \rG$. Since $\infty$ is a cusp of $\rG$, there is a parabolic element in $\rG\backslash\{I\}$  whose lower left entry vanishes. Such an element necessarily equals $\tmt{1}{h}{0}{1}$ for some nonzero $h$. Using that $0$ is a cusp of $\rG$, we similarly obtain $\tmt{1}{0}{h'}{1}\in \rG$ for some nonzero $h'$. For any integer $n$, $$\tmt{a}{b}{c}{d}\tmt{1}{0}{h'}{1}^n=\tmt{a+nh'b\,\,}{b}{c+nh'd\,\,}{d}.$$ Since $h'd\neq0$, there is $n$ such that $|c+nh'd|\leq|h'd|$, namely the integer part of $-c/h'd$. By Lemma 1.7.3 of~\cite{Miyake}, either $c+nh'd=0$ or $|c+nh'd|\ge|h|^{-1}$. If $c+nh'd=0$, by part 2 of Theorem 1.5.4~ in ~\cite{Miyake} with $x=\infty$ and $\sigma$ equal to the identity, we have $|d|=1$. If $|c+nh'd|\ge|h|^{-1}$, by our choice of $n$, we get $|d|\geq|hh'|^{-1}$. So $\inf_{\gamma\in \rG} |j(\gamma,0)|\ge\min\{1,|hh'|^{-1}\} > 0$, as desired.

Now we show the claim for a general $\cu$. We shall move the cusp to the origin by means of the translation $B_\cu(\tau) = \tau - \cu.$ Observe that $0$ is a cusp of the Fuchsian group $B_\cu \rG B_\cu^{-1}$. Indeed, let $\delta\in \rG$ be a parabolic element such that $\delta\cu = \cu$. Then $B_\cu\delta B_\cu^{-1}$ is parabolic and $B_\cu\delta B_\cu^{-1}0 = B_\cu\delta\cu = B_\cu\cu = 0$. Similarly, $\infty$ is a cusp of $B_\cu \rG B_\cu^{-1}$. If $0$ and $\infty$ were equivalent cusps of this new group, there would exist $\delta\in B_\cu \rG B_\cu^{-1}$ such that $\delta0 = \infty$. Then we would have $B_\cu^{-1}\delta B_\cu\in \rG$ and $B_\cu^{-1}\delta B_\cu \cu = B_\cu^{-1}\delta0 = B_\cu^{-1}\infty = \infty$, so $\cu$ would be equivalent to $\infty$ as a cusp of $\rG$, a contradiction. We have established $0$ and $\infty$ are inequivalent cusps. From the case we have already proven, 
\begin{equation}\label{lower_bound_j}
\inf_{\gamma\in B_\cu \rG B_\cu^{-1}}|j(\gamma,0)|>0.
\end{equation}
Now let $\gamma\in \rG$ and $\tilde\gamma = B_\cu\gamma B_\cu^{-1}$. Then $\gamma\cu \neq\infty$ and $\tilde \gamma B_\cu = B_\cu\gamma$. Hence, using the definition of $j$ and~\eqref{eq:property-j}, we find that $ j(\tilde\gamma, 0) = j(\gamma, \cu)$.  Combining this with inequality~\eqref{lower_bound_j}, we conclude the proof.
\end{proof}

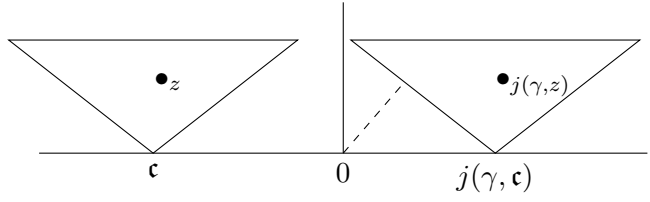
\begin{figure}
\centering
\begin{tikzpicture}
\draw (-4,0)--(4,0);
\draw (0.1,1.5) node[anchor=north]{\nonumber}
  -- (2,0) node[anchor=north]{$j(\gamma,\cu)$}
  -- (3.9,1.5) node[anchor=south]{\nonumber}
  -- cycle;
\draw (2.5,1.2) node[anchor=north]{$\bullet_{j(\gamma,z)}$};
\draw (0,0) node[anchor=north]{${0}$}--(0,2);  
\draw[dashed] (0,0)--(0.82,0.96);
\draw (-2.5,0) node[anchor=north]{${\cu}$}
  -- (-4.4,1.5) node[anchor=north]{\nonumber}
  -- (-0.6,1.5) node[anchor=south]{\nonumber}
  -- cycle;
\draw (-2.3,1.2) node[anchor=north]{$\bullet_z$};
\end{tikzpicture}
\caption{If $z$ lies in a given triangle $\mathcal S(\cu,v_0,K)$, then $j(\g,z)$ lies in a similar triangle, which helps to estimate its distance to the origin.}
\end{figure}

\begin{cor}\label{lowerboundjgammaz}
Let $\cu$ and $\infty$ be inequivalent cusps of the Fuchsian group $\rG$. Let $K$ and $v_0$ be positive real numbers. Then $|j(\gamma,z)|$ is bounded from below, for any $\gamma\in \rG$ and any $z\in\mathcal S(\cu, v_0, K)$.
\end{cor}

\begin{proof}
When $z$ varies in $\mathcal S(\cu,v_0,K)$, the point $j(\gamma,z)=cz+d$ varies in a similar triangle to $\mathcal S(\cu,v_0,K)$, with a vertex at $j(\gamma,\cu)$ (see Figure 2). In particular, $j(\gamma,z)$ lies between the straight lines through $j(\gamma, \cu)$ with slopes $1/K$ and $-1/K$. So the distance from $j(\gamma,z)$ to the origin is greater than the distance from some of these straight lines to the origin. From trigonometry, the latter is $|j(\gamma,\cu)|/\sqrt{K^2+1}$. Therefore $$|j(\gamma,z)|\ge\frac{|j(\gamma,\cu)|}{\sqrt{K^2+1}}.$$
The claim now follows from Lemma~\ref{lowerboundj}.
\end{proof}

We now give a proof of the bound of order $O(n^{k+2\alpha})$ in Theorem~\ref{maintheorem}.
\subsubsection{Proof of part (a) of Theorem~\ref{maintheorem}: holomorphic vvaf}
We begin with the case $k+2\alpha\ge0$. Let $\mathcal S(\cu,v_0,K)=\{u+iv\in\h:v<v_0\text{ and }|u-\cu|\leq Kv\}$. There exists constants $v_0,v_1,K$ and a fundamental domain $\mr{F}_\rG$ that is contained in a finite union of sets of type $\mathcal S(\cu,v_0,K)$, where $\cu$ is a cusp of $\rG$ that is not equivalent to $\infty$, and of a set $\{u+iv\in\h:0\leq u<h\text{ and }v>v_1\}$. Take $\tau$ in $\h$ such that $\mr{Im}\tau=y<v_1$. There exists $z=u+iv$ in $\mr{F}_\rG$ such that $\gamma z=\tau$. From inequality~\eqref{bound_X},
$$y^{k+2\alpha}\|\X(\tau)\|\ll|cz+d|^{-k-2\alpha}(1+4|z|^2)^\alpha v^k\|\X(z)\|$$

From the fact that $\X|_kA_\cu$ is bounded near $\infty$, one can show that $v^k\|\X(z)\|$ is bounded in any set of type $\mathcal S(\cu,v_0,K)$. Therefore, in such a set,
\begin{align*}
y^{k+2\alpha}\|\X(\tau)\|\ll |cz+d|^{-k-2\alpha}
\end{align*}
By Corollary~\ref{lowerboundjgammaz}, $|cz+d|$ has a lower bound independent of $\gamma$ and $z$, for any $z$ in $\mathcal{S}(\cu, v_0, K)$. This implies that $y^{k+2\alpha}\X(\tau)$ is bounded, since we are assuming that the exponent $-k-2\alpha$ is negative.

It remains to consider the case in which $0\leq u<h$ and $v>v_1$. Now $\frac{1+4|z|^2}{v^2}$ is bounded, so that
$$y^{k+2\alpha}\|\X(\tau)\|\ll\left(\frac{v}{|cz+d|}\right)^{k+2\alpha}\leq |c|^{-k-2\alpha}.$$
By Lemma 1.7.3 of~\cite{Miyake} and the hypothesis that the exponent $-k-2\alpha$ is negative, this has an upper bound unless $\gamma = \tmt{1}{h}{0}{1}^n$. But $\gamma = \tmt{1}{h}{0}{1}^n\implies y=v>v_1$, contradicting our choice of $\tau$. Therefore, $y^{k+2\alpha}\|\X(\tau)\|$ is bounded if $y$ is sufficiently small.  As in the proof for cusp forms, one obtains the bound for the Fourier coefficients of $\X$.

\subsubsection{Proof of part (c) of Theorem~\ref{maintheorem}: case $k+2\alpha<0$}
To treat the case $k+2\alpha<0$, we work with a fundamental domain that is contained in a finite union of sets of the type $\mathcal S(\cu,v_0,K)$. In such a region, we employ inequality~\eqref{bound_product} together with the bound $\X(z)\ll v^{-k}$, and get
$$y^{k/2+\alpha}\|\X(\tau)\|\ll v^{-k/2-\alpha}\ll 1,$$
since $v$ is bounded. This implies that $\|\X(\tau)\|\to0$ as $y\to 0$. Therefore, each component of $\X(\tau)$ approaches to $0$ as $y\to 0$. Multiplying the $i^{\text{th}}$-component of $\X$ by $\tq^{1-\mu(\lambda_{i})}$, we get a power series in $\tq$ which approaches to $0$ as $|\tq|\to1$. By the maximum principle applied to circles with radius close to $1$, we conclude that each component of $\X(\tau)$ vanishes.\qed


\section{Growth for logarithmic vector-valued automorphic forms}\label{se:logvvaf}
In the previous section we studied growth of admissible case. Recall that our definition included moderate growth for this case. We mentioned in Remark~\ref{rem:unitary} that such a growth implies that the representation  has only unitary eigenvalues. We initiated the discussion about the logarithmic vvaf in Section~\ref{sec:vvaf}. Now, in this general setting, we are not imposing the moderate growth condition. In this case, we are assuming that all the eigenvalues of the image of each parabolic element are unitary.

\subsection{Polynomial-growth of the representation } One of the big advantages of assuming $\rho$ admissible was that, $\norm{\rho(t_{\cu}^n)}=O_{\cu}(1),$ for any cusp $\cu$ of $\rG$ and $n\in \mathbb{Z}$. However, the same may not hold when $\rho$ is not admissible. We have the following lemma to overcome that obstacle.
\begin{lem}\label{lem:parabolic} For any integer $n\neq 0$, and any parabolic element $t_{\cu}\in \rG,$ we have the following estimate
\[\norm{\rho(t_{\cu}^n)}\ll_{\cu,m} |n|^{m-1}.\]
\end{lem}
\begin{proof} Due to the assumption, we may assume that $\rho(t_{\cu})$ is conjugate to a matrix in the Jordan canonical form. Now it is enough to bound norms of the corresponding Jordan blocks. Let $J_{m_t,\lambda_t}$ be one of such blocks. We can write
\[J_{m_t,\lambda_t}^{n}=\lambda_t^{n}(I_{m_t}+N)^{n}=\lambda_t^n\sum_{0\leq i\leq m_t}\binom{n}{i}N^i,\]
because $N^i=0$ for any $i\geq m_t.$ In particular, we then have
\[\norm{J^n_{m_t,\lambda_t}}\ll_{m_t} n^{m_t-1},\]
because $|\lambda|=\norm{N}=1$ and $\sum_{0\leq i\leq m_t-1} \binom{n}{i} \ll_{m_t} n^{m_t-1}.$ The result now follows by varying the Jordran blocks.
\end{proof}
We start with considering the decomposition of $\gamma$ given by Beardon, as discussed in Section~\ref{structure}. 

\begin{lem}\label{lem:power} The product of the powers of the parabolic elements coming in the word $\gamma=C_1C_2\cdots C_s$ is at most $\norm{\gamma}^{\alpha_1},$ for some constant $\alpha_1$ depending on $\rG.$

\end{lem}
\begin{proof} It follows from Lemma~\ref{lem:beardon} that there exists a constant $c$ such that whenever $|C_i|>c,$ we have $C_i$ is a product of a power of $t_{\cu}$, (where $\cu$ is a cusp of $\rG$) with an element coming from a finite subset $\rG^{*}$ of $\rG.$ In this case, $\norm{C_i} \gg $ the power of $t_{\cu}$ appearing in $C_i.$ On the other hand, the number of $C_i$ with $|C_i|\leq c$ is bounded by $|\rG^{*}|^{c}=O(1).$ In particular, all of the powers of parabolic elements appearing in such $C_i$’s are also $O(1)$. Therefore, the desired product of the powers of the parabolic elements coming in $\gamma$ is bounded by $O\left(\prod_{i,\mid C_i\mid >c} \norm{C_i}\right).$ The proof is now complete by~\cite[Theorem 2]{Beardon1975}.
\end{proof}
As a consequence we have the following growth result on $\rho.$ It is not hard to see that, we do not have such a nice growth if image of some parabolic element have non unitary eigenvalues.  
\begin{lem}\label{lem:polygrowth} Let $\gamma:=\tmt{a}{b}{c}{d} \in \rG$ be an arbitrary element. Then we have that
$\norm{\rho(\gamma)}\ll (a^2+b^2+c^2+d^2)^{\alpha'},$
for some constant $\alpha'$ depending on $\rG.$
 \end{lem}

\begin{proof} 
We first consider the decomposition $\gamma=\prod_{i=1}^{s}C_i$ as given by Beardon, and we obtain $\norm{\rho(\gamma)}\leq \prod_{i=1}^{s}\norm{\rho(C_i)}.$ Once again, since $\rG^*$ is finite, the terms with $| C_i|< c$ do not contribute much. On the other hand, following Lemma~\ref{lem:beardon}, $|C_i|>c$ implies that there exists a cusp $\cu_i$ of $\rG$ such that $C_i$ is a product of $t_{\cu_i}^{n_i}$ with an element from $\rG^{*}.$ In particular, it follows from Lemma~\ref{lem:parabolic} that 
$$\norm{\rho(C_i)}\leq M_{\rG}\norm{\rho(t_{\cu_i}^{n_i})}\ll_{\cu_i,m}M_{\rG}|n_i|^{m-1},$$
where $M_{\rG}$ is the maximum of $\norm{\rho(\gamma)}_{\gamma \in \rG^{*}}.$ Now each such $\cu_i$ is in fact a vertex of $\widehat{\mr{D}}_\rG,$ and also the rank $m$ is fixed, hence we can actually write $\norm{\rho(C_i)}\ll M_{\rG}|n_i|^{m-1}.$ We then have the following estimate 
\[\norm{\rho(\gamma)}\leq \prod_{i=1}^{s}\norm{\rho(C_i)} \ll \prod_{i,| C_i| >c}\norm{\rho(C_i)} \ll M_{\rG}^{s}\prod_{\substack{i,\mid C_i\mid >c\\ t_{\cu_i}^{n_i}\in C_i}}| n_i|^{m-1}.\]
Now we get the desired bound by applying Lemma~\ref{lem:power} and the bound $s=O(\log(a^2+b^2+c^2+d^2))$ from~\cite[Theorem 2]{Beardon1975}.

\end{proof}

\subsection{Recipe to bridge two certain regions in $\mathbb{H}$}
We need another ingredient to finish our preparation for the logarithmic case of Theorem~\ref{maintheorem}. As in the proof of admissible case, we shall need a relation between $|cz+d|^2$ and $c^2+d^2$, where $(c,d)$ is the last row of a matrix in $\rG$. However, Lemma~\ref{lem:polygrowth} gives a bound in terms of $a^2+b^2+c^2+d^2$. For this reason it is useful to have an inequality of the form $a^2+b^2 \ll c^2+d^2$. We shall shortly see that there exists a region where vvaf has the desired growth, and the following result gives us an element of $\rG,$ which serves as a bridge to the region $1\leq x \leq h, 0<y<1.$

\begin{lem} Let $\gamma_0 = \tmt{a_0}{b_0}{c_0}{d_0}$ with $c_0\neq0$. Let $\gamma\in \mr{PSL}_2(\R)$ have rows $r_1$ and $r_2$. Let $\tilde\gamma = \gamma_0\gamma$ have rows $\tilde r_1$ and $\tilde r_2$. Then either $\norm{r_1}\le\max\left\{1, 2\left|\frac{d_0}{c_0}\right|\right\}\norm{r_2}$ or $\|\tilde r_1\|\le2\frac{|a_0|+|b_0|}{|c_0|}\|\tilde r_2\|,$ where  $\norm{\cdot}$ of the rows are defined as the usual norm in $\mathbb{R}^2.$ 
\end{lem}

\begin{proof} Assume $\|r_1\|\ge\max\{1, 2\left|\frac{d_0}{c_0}\right|\}\|r_2\|$. From the definition of matrix multiplication, $\tilde r_1 = a_0r_1+b_0r_2$ and $\tilde r_2 = c_0r_1+d_0r_2$. From the triangle inequality, 
\begin{equation}\label{r_tilde_bound}
    \|\tilde r_1\|\le|a_0|\cdot\|r_1\|+|b_0|\cdot\|r_2\|\le(|a_0|+|b_0|)\|r_1\|.
\end{equation}
Applying the triangle inequality, the hypothesis $\|r_1\|\ge2\left|\frac{d_0}{c_0}\right|\|r_2\|$ and~\eqref{r_tilde_bound}, gives us
$$ \|\tilde r_2\| \ge|c_0|\cdot\|r_1\|-|d_0|\cdot\|r_2\| \ge\frac{|c_0|}{2}\|r_1\| \ge\frac{1}{2}\cdot\frac{|c_0|}{|a_0|+|b_0|}\|\tilde r_1\|,$$
which is equivalent to the stated inequality.
\end{proof}

\subsection{Proof of Theorem~\ref{maintheorem}: Logarithmic case}
Let us first consider the case when $\X(\tau)$ is a cusp form. We claim that there exists $\gamma_0 = \tmt{a_0}{b_0}{c_0}{d_0}\in \rG$ such that $c_0>0$ and $d_0>0$. The existence of a matrix $\g_1=\tmt{a_1}{b_1}{c_1}{d_1}$ such that $c_1\neq0$ follows from~\cite[Theorem 1.5.4]{Miyake} and the assumption that $\rG$ is of the first kind. Since we are in $\mr{PSL}_2(\R)$ rather than in $\mr{SL}_2(\R)$, we may assume $c_1>0$. Let $n$ be so large that $d_1+nhc_1>0$. Then we can take $\gamma_0=\gamma_1t_\infty^n=\tmt{a_1}{b_1+nha_1}{c_1}{d_1+nhc_1}$. From now on we fix such a $\gamma_0$.

We consider $\mr{F}_{\rG}$ to be a fundamental domain of $\rG$ that is union of the sets of type $\mathcal{S}(\cu,v_0,K).$ We now consider 
\[\rG_1=\left\{\tmt{a}{b}{c}{d}\mid \|r_1\|\le\max\left\{1, 2\left|\frac{d_0}{c_0}\right|\right\}\|r_2\|\right\},\]
and denote $\mathbb{H}_1=\rG_1\mathbb{F}_{\rG}.$ The previous lemma, and~(\ref{another_bound}) implies that for any $\tau\in \mathbb{H}_1,$
\begin{equation}\label{eqn:logcusp}
y^{k/2+\alpha'}\norm{\X(\tau)}\ll v^{k/2-\alpha'}\norm{\X(z)},
\end{equation}
where $\tau=\gamma z$ for some $\gamma \in \rG_1$ and $z$ is in the union of $\mathcal{S}(\cu, v_0,K)$'s. Since $\X$ is a cusp form, all the $\widetilde{q}_{\cu}$-expansions associated to $\X|_{k}A_{\cu}$ decay exponentially as the imaginary part of $A_\cu^{-1}z=\frac{1}{\cu-z} \to \infty.$ On the other hand, $|\log \widetilde{q}_{\cu}|$ grows polynomially, as $\text{Im}\left(\frac{1}{\cu-z}\right) \to \infty$ and $\text{Re}\left(\frac{1}{\cu-z}\right)$ can be taken to be bounded. The maximum power of $\log \widetilde{q}_{\cu}$ appearing in the logarithmic expansion is at most $\dim(\rho),$ and it is clear that an extra polynomial factor does not affect the exponential decay of the logarithmic expansion. In particular, $\X|_kA_{\cu}$ decays exponentially as $y\to \infty.$ Now arguing similarly as in Theorem~\ref{maintheorem}, that is, by comparing $\X$ with $\X|_kA_{\cu},$ we get $y^{\beta}\norm{\X(z)}=O(1)$ for any $z\in S(\cu, y_0, K)$ and real $\beta.$ We now get $\norm{\X(\tau)}\ll y^{-k/2-\alpha'}$ for any $\tau \in \mathbb{H}_1$ by taking $\beta=\alpha'+\frac{k}{2}$ in~(\ref{eqn:logcusp}).

Now we compare $\X(\tau)$ with $\X(\gamma_0\tau)$ to see the growth in $\mathbb{H}\setminus \mathbb{H}_1.$ Let $\tau=x+iy \in \mathbb{H}\setminus \mathbb{H}_1,$ then using the functional equation,
\begin{align*}
\|\X(\tau)\|&=|c_0\tau+d_0|^{-k}\|\rho(\gamma_0)^{-1}\X(\gamma_0\tau)\|\\
&\ll|c_0\tau+d_0|^{-k}(\mr{Im}\gamma_0\tau)^{-k/2-\alpha'}\\
&\ll|c_0\tau+d_0|^{-k}\left(\frac{y}{|c_0\tau+d_0|^2}\right)^{-k/2-\alpha'}\\
&\ll|c_0\tau+d_0|^{2\alpha'}y^{-k/2-\alpha'}\\
&\ll y^{-k/2-\alpha'},
\end{align*}
for any $0\leq x \leq h$ and $0<y<1.$ In particular, $|\X_i(\tau)| \ll y^{-k/2-\alpha'} $ for all $0\leq i\leq m-1.$ Then it follows inductively from~(\ref{eqn:matrix}) that $|\tilde{h}_{\lambda,j}(\tau)|\ll y^{-k/2-\alpha'},$ and in particular
\begin{equation}\label{eqn:lvaf bound}
    |h_{i,j,\lambda}(\tau)|, \norm{h_{\lambda,j}(\tau)}\ll y^{-k/2-\alpha'},
\end{equation}
for all $1\leq j\leq m(\lambda)-1,$ any eigenvalue $\lambda$ of $\rho(t_{\infty}),$ and $0\leq x\leq h.$ Now note that
\[\norm{\X_{[\lambda,j,n]}}\ll \frac 1 h\int_{0}^{h}\norm{h_{\lambda,j}(x+iy){\tq}^{(-n-\mu(\lambda))}} dx,~\forall ~0\leq j\leq m(\lambda)-1.\]
In particular, we then have
\[\norm{\X_{[\lambda,j,n]}} \ll y^{-k/2-\alpha'} e^{2\pi y(n+\mu(\lambda))/h}, ~\forall ~0\leq j\leq m(\lambda)-1.\]
Now, taking $y=\frac{1}{n+\mu(\lambda)}$ and setting $\alpha$ to be $\alpha'$ we get the desired result.

Let us now consider the holomorphic case. Similarly as in the previous case, we shall first show that $\X$ has polynomial-growth in $\mathbb{H}_1.$ Consider a fundamental domain that is covered by finitely many sets of type $\mathcal S(\cu,y_0,K),$ where $\cu$ is a finite cusp of $\rG$ not equivalent to $\infty,$ and a region of type $\mathcal S=\{z \in \mathbb{H}\mid 0\leq u \leq h, v> v_0\}.$ From~\eqref{bound_X} we have for any $\tau \in \mathbb{H}_1$ that
\begin{equation}\label{eq:log ineq}
y^{k+2\alpha'+m}\|\X(\tau)\|\ll|cz+d|^{-k-2\alpha'-2m}(1+4|z|^2)^{\alpha'} v^{k+m}\|\X(z)\|
\end{equation}
where $\tau=\gamma z,$ for some $\gamma \in \rG$ and $z=u+iv$ lying in one of the sets $\mathcal S(\cu,y_0,K).$ Following the arguments given in the previous case, we have $v^{k+m}\|\X(z)\|=O(1).$ This is because, all $\widetilde{q}_{\cu}$-expansions of $\X|_kA_{\cu}$ are bounded near $\infty$ and the extra $\log$ factor grows like $\text{Im}\left(\frac{1}{\cu-z}\right)\sim v^{-1}.$ It now follows immediately from Lemma~\ref{lowerboundjgammaz} that $y^{k+2\alpha}\norm{\X(\tau)}$ is bounded whenever $z$ is in one of the sets $\mathcal S(\cu,v_0,K),$ where $\alpha=\alpha' + m$. On the other hand if $z$ in $\mathcal S$, we have to be a little more careful because of the extra unbounded $\log$ factors coming in the Fourier expansion. From~\eqref{eq:log ineq} it follows that
\begin{align*}
y^{k+2\alpha'+m}\norm{\X(\tau)} &\ll v^{k+m}|c z+d|^{-k-2\alpha'-2m}(1+4|z|^2)^{\alpha'}\norm{\X(z)} \\
& \ll v^{k+2m+2\alpha'}|c z+d|^{-k-2\alpha'-2m}v^{-m}\norm{\X(z)},
\end{align*}
as $\frac{1+4|z|^2}{v^2}$ is bounded. Now note that $v^{-m}\norm{\X(z)}=O(1)$ in $\mathcal{S}$ because the $h_{\lambda,j}$'s from~(\ref{fourier_expansion_X_i}) are bounded in $\mathcal{S}, |\log z|=O(y)$ and the maximum power of $\log$ appearing in the expansion of $\X(\tau)$ goes up to at most $m.$ Thus, if $k+2\alpha'+2m\ge0$ and $y<v_0$, we have $y^{k+2\alpha'+m}\|\X(\tau)\|\ll1$ as in the admissible case.

We now need to relate $\mathbb{H}_1$ with $\mathbb{H}$ and for that we are again going to rely on the comparison of $\X(\tau)$ with $\X(\gamma_0\tau).$ Note that we have estimated $\X(\tau)$ only at the points of $\h_1$ with small imaginary part. So we need to ensure that $\mr{Im}\g_0\tau=\frac{y}{|c_0\tau+d_0|^2}$ is small. This happens when $0\le x\le h$ and $y$ is small, because then $|c_0\tau+d_0|\ge\mr{Re}(c_0\tau+d_0)\ge d_0$ by our choice of $\g_0$. Arguing similarly as in the previous case, we have that for any $\tau=x+iy\in \mathbb{H}\setminus \mathbb{H}_1$,
\[\norm{\X(\tau)}\ll y^{-k-2\alpha'-m},\]
provided $\tau$ is bounded. Performing the integration, we get the desired result by the same choice of $\alpha.$\\

For part $(c)$, we may assume that $k+2\alpha'+m<0$. If we have also $k+2\alpha'+2m\ge0$, the previous argument gives $\|\X(\tau)\|\ll y^{-k-2\alpha'-m}$, so $\X(\tau)\to0$ as $y\to 0$ with $0\le x\le h$. If $k+2\alpha'+2m<0$, let $\tilde\alpha$ solve the equation $k+2\tilde\alpha+2m=0$. Then we can apply the same argument with $\tilde\alpha$ instead of $\alpha'$. We get $\|\X(\tau)\|\ll y^{-k-2\tilde\alpha-m} = y^m$, which also approaches to $0$ as $y\to0$. Now we obtain $\X\equiv0$ as in the admissible case. \qed


\section{Growth of representations}\label{sec:rmks} In Theorem~\ref{maintheorem} we required a unitary condition on the eigenvalues in order to get a Fourier expansion. One of the consequence of this condition was that, the corresponding representation has polynomial-growth. Recall from Section~\ref{se:logvvaf} that, by polynomial-growth, we mean existence of a constant $\alpha$ such that $\norm{\rho(\gamma)}\ll \norm{\gamma}^{\alpha},$ for any $\gamma$ in the group $\rG.$ In this section, we want to see when a given representation has polynomial-growth, and what happens to the growth of vector-valued meromorphic functions which satisfies the functional equation $\X|_k\g = \rho(\g) \X, \forall \g \in \rG,$ with respect to the given representation $\rho$. 

\subsection{On polynomial-growth} We have the following criteria for polynomial-growth of $\rho,$ which basically says that it is enough to look over only the set of parabolic elements. \begin{prop}\label{prop:criteria} Let $\rG$ be a Fuchsian group of the first kind and $\rho:\rG \to \mr{GL}_m(\C)$ be a representation. Then $\rho$ has polynomial-growth if and only if all eigenvalues of image of each parabolic element are unitary. 
\end{prop}
\begin{proof} Suppose that all eigenvalues of image of each parabolic element are unitary, then we get the polynomial-growth immediately from Lemma~\ref{lem:polygrowth}. Now for the other direction, take $\gamma\in\rG$ to be a parabolic element such that at least one eigenvalue of $\rho(\gamma)$ is non-unitary. Note that $\gamma$ is conjugate to an element of the form $\tmt{1}{mh}{0}{1}$ and in particular $\norm{\gamma^n}\ll n.$ It now follows from the computation with Jordan canonical form as in the proof of Lemma~\ref{lem:power} that $\norm{\rho(\gamma^n)}\geq r^n,$ for some $r>1.$ Here $r$ can be taken to be modulus of any non-unitary eigenvalue of $\rho(\gamma)$. This gives a contradiction.
\end{proof}

We now have an interesting consequence which says that the polynomial-growth is preserved under induction, restriction and isomorphism.
\begin{cor}\label{cor:more} Let $\rH\subseteq \rG$ be two Fuchsian groups of the first kind and $\rH$ has finite index in $\rG.$ Then $\rho$ has polynomial-growth if and only if the induced representation $\tilde \rho:=\mr{Ind}_{\rH}^{\rG}\rho$ has polynomial-growth.
\end{cor}
For the definition and the details on induced representation and their associated vector-valued automorphic forms see Sections 3 and 4 of~\cite{Bajpai2019}. Let us recall the definition of $\tilde\rho$. Write $\rG= \g_{1}\mr{H} \cup \g_{2}\mr{H} \cup \cdots \cup \g_{d}\mr{H},$ where $d$ is the index of $\rH$ in $\rG$. Without loss of generality we may assume that $\g_1=1$. The representation $\rho :\mr{H} \rightarrow \mr{GL}_{m}(\C)$ can be extended to a function on all of $\rG$, \ie $\rho: \rG \rightarrow \mr{M}_{m}(\C)$ by setting $\rho(x) =0,  \forall x \notin \mr{H}$ where  $\mr{M}_{m}(\C)$ is the set of all $m\times m $ matrices over $\C$. The induced representation $\tilde\rho: \rG \rightarrow \mr{GL}_{dm}(\C)$ is defined by \begin{equation}\label{eq:indrep} \tilde\rho(x) = \left( \begin{array}{cccc}
\rho(\g_{1}^{-1} x \g_{1}) & \rho(\g_{1}^{-1} x \g_{2}) & \ldots  &\rho(\g_{1}^{-1} x \g_{d})\\
\rho(\g_{2}^{-1} x \g_{1}) & \rho(\g_{2}^{-1} x \g_{2}) & \ldots & \rho(\g_{2}^{-1} x \g_{d}) \\
\vdots & \vdots & \ddots& \vdots \\
\rho(\g_{d}^{-1} x \g_{1}) & \rho(\g_{d}^{-1} x \g_{2}) & \ldots  &\rho(\g_{d}^{-1} x \g_{d})\\
\end{array} \right), \quad \forall x \in \rG. \end{equation} 
Now for any $x \in \rG$ and $\forall 1\leq i \leq m,$ there exists a unique $1\leq j \leq m$ such that $\rho(\g_{i}^{-1} x \g_{j}) \neq 0$. Therefore, exactly one nonzero $m \times m$ block appear in every row and every column of~\eqr{indrep}.
\begin{proof}[Proof of Corollary~\ref{cor:more}]  Due to Proposition~\ref{prop:criteria} we now know a representation has polynomial-growth if and only if, every eigenvalue of image of each parabolic element is unitary. We shall prove this result with respect to this unitary property.

Restriction invariant is an immediate consequence. Now if $\rho_1$ and $\rho_2$ are isomorphic representations, then $\rho_1(\gamma)$ is conjugate to $\rho_2(\gamma)$ for each element $\gamma \in \rG.$ In particular, all eigenvalues of $\rho_1(\gamma)$ are unitary if and only if, all eigenvalues of $\rho_2(\gamma)$ are unitary.

For the induction invariance, let $\tilde \rho$ be the induced representation of $\rho$ with respect to a choice of the coset representatives $\gamma_1,\ldots, \gamma_{d}$ of $\rH$ in $\rG$, where $d$ is the index of $\rH$ in $\rG.$ Let $\g\in \rG$ be a parabolic element. Then, for each $i$, some non-trivial power of $g_i=\gamma_i^{-1}\g\gamma_i$ is in $\rH$. This can be shown from the existence of $n_{i,1}$ and $n_{i,2}$ such that $g_i^{n_{i,1}}\rH=g_i^{n_{i,2}}\rH$, say the $n_i^{\text{th}}$-
power. Then $\gamma_i^{-1}\g^N\gamma_i\in \rG$ for each $i$, where $N=\text{lcm}\{n_i\}$. In particular, $\tilde \rho(\g^N)$ is a block diagonal matrix whose each blocks are of the form $\rho(\gamma_i^{-1}\g^N\gamma_i)$. If $\rho$ has polynomial-growth, then each such block has unitary eigenvalues, and in particular $\tilde \rho(\g^N)$ has unitary eigenvalues. Moreover, the eigenvalues of $\tilde\rho(\g)$ are $N^{\text{th}}$ roots of the eigenvalues of $\tilde\rho(\g^N)$, as one can see from the Jordan canonical form of $\tilde\rho(\g)$. Therefore, all the eigenvalues of $\tilde{\rho}(\g)$ are also unitary, and hence $\tilde \rho$ has polynomial-growth. 

On the other hand suppose that $\tilde \rho$ has polynomial-growth. Take an element $\gamma_0\in\rH,$ and it is enough to show that every eigenvalue of $\rho(\g_0)$ is unitary. It follows from the discussion in previous paragraph that, there exists $N$ such that $\tilde\rho(\g_0^N)$ is a block diagonal matrix. Moreover, one of the block is $\rho(\g_0^N)$ since one of the representative can be taken to be the identity element of $\rH.$ In particular, $\rho(\gamma_0^N)$ has only unitary eigenvalues, and so does $\rho(\gamma_0),$ as desired.
\end{proof}

\subsection{On a sharp polynomial-growth for finite index subgroups of $\text{PSL}_2(\Z)$} In Lemma~\ref{lem:polygrowth} we had a polynomial-growth on the representation  involving all the entries, while in Lemma~\ref{lem:alpha} the bound only involved the bottom row. The difference is that, in the admissible case, for any parabolic element $\gamma\in \rG,~\norm{\rho(\gamma^n)}$ is bounded irrespective to $n.$ To improve Lemma~\ref{lem:polygrowth} we need to control the number of times parabolic elements are appearing when we decompose an element of $\rG.$ For example, when we take $\rG$ to be a finite index subgroup of $\mathrm{PSL}_2(\Z),$ we have a better control. 
\begin{prop} Let $\rG$ be the finite index subgroup of $\mr{PSL}_2(\mathbb{\Z}),$ and $\rho$ be a representation of $\rG$ such that, image of each parabolic element has only unitary eigemvalues. Then we have, 
$$\norm{\rho(\gamma)}\ll (c^2+d^2)^{\alpha}\max\left\{[|a/c|]^{m-1},1\right\},~\forall \gamma \in \rG.$$
\end{prop}
\proof First consider the induced representation $\tilde{\rho}$ of $\rho$ to $\mr{PSL}_2(\mathbb{\Z}).$ It follows from the proof of Corollary~\ref{cor:more} that, image of each parabolic element under this induced representation have only unitary eigenvalues. Take an element $\gamma \in \rG\subseteq \mathrm{PSL}_2(\Z).$ We start with writing $\gamma=(st^{l_{v+1}})(st^{l_v})\cdots (st^{l_1})(st^{l_0})$ with $s=\tmt{0}{-1}{1}{0},t=\tmt{1}{1}{0}{1},$ as in~\cite[Lemma 3.1]{KM2012}. Using Corollary 3.5 and the estimate of $l_{v}$ from Lemma 3.1 of~\cite{KM2012}, we get $\norm{\tilde{\rho}(\gamma)} \ll (c^2+d^2)^{\alpha}\max\left\{[|a/c|]^{m-1},1\right\}.$ Since $\norm{\rho(\gamma)}\leq \norm{\tilde{\rho}(\gamma)}$, the result follows.\qed

Therefore we indeed have a better growth for finite index subgroups of $\mathrm{PSL}_2(\mathbb{Z})$ at least when $c\neq 0,$ in the sense that the bound does not involve one of the entries of $\gamma.$ When $c=0,$ the element $\gamma$ is parabolic. In that case, one can get a bound from Lemma~\ref{lem:parabolic}.

\subsubsection{Example of a representation with non polynomial growth}\label{ex:non-poly} There exists a $\rG$ and a representation $\rho$ of it such that $\rho$ does not have polynomial-growth. For instance, consider $\rG=\text{PSL}_2(\Z)$ and $\rho:\rG\to\text{GL}_3(\C)$ given by
\[\rho(s)=\begin{psmallmatrix}
    a & -(a+1) & 1 \\
    a-1  & -a & 1\\
    0 & 0 & 1 \\
  \end{psmallmatrix},~\rho(t)=\mr{diag}(\lambda_1,\lambda_2,\lambda_3),\]
  where $a,\lambda_1,\lambda_2,\lambda_3$ are yet to be chosen.
See section 2.1 of~\cite{CHMY} for a proof about why $\rho$ is a representation, provided that $\lambda_1\lambda_2=-\lambda_3^2, \frac{\lambda_1\lambda_2}{(\lambda_1-\lambda_2)^2}=-a^2$ and $\frac{1}{\lambda_1\lambda_2(\lambda_1-\lambda_2)}=a.$ We can make sure about these conditions hold by taking $\lambda_1,\lambda_2,\lambda_3$ as follows: Take $\lambda_3=1$ and $\lambda_1,\lambda_2$ in such a way, so that $\lambda_1\lambda_2=-1$ and $\lambda_1-\lambda_2=-\frac{1}{a}.$ We want to be able to make one of $\lambda_1$ or $\lambda_2$ non unitary, and we can now ensure that by taking any purely imaginary $a.$

\subsubsection{Example of a representation with polynomial growth} Consider $$ \X(\tau)= \frac{1}{\eta(\tau)} \begin{pmatrix} \theta_{2}(\tau)\\ \theta_{3}(\tau)\\ \theta_{4}(\tau)  \end{pmatrix},$$
where $\theta_2(\tau),\theta_3(\tau),\theta_4(\tau)$ and $\eta(\tau)$ are well known weight $1/2$ scalar-valued modular forms. For a complete description of these functions, the reader may refer to~\cite{Knopp}. It turns out that $\X(\tau)$ is a vector-valued modular function of $\mr{PSL}_2(\Z)$ and representation  $\rho$ where $\rho: \mr{PSL}_2(\Z) \rightarrow \mr{GL}_3(\C)$ is a rank 3 representation of $\G(1)$ given by
\begin{equation} \rho(s)= \begin{pmatrix}  0 &0&1 \\ 0&1&0 \\ 1&0&0 \end{pmatrix} \ \mr{and}\,\ \rho(t)=  \begin{pmatrix}  \exp(\frac{\pi \Ii}{6}) &0& 0 \\ 0& 0 &\exp(-\frac{\pi \Ii}{12}) \\ 0&\exp(-\frac{\pi \Ii}{12})&0 \end{pmatrix} \nonumber\,.\end{equation} 
To see whether $\rho$ has polynomial-growth, it is enough to check the eigenvalues of $\rho(t).$ In this case, they are given by $\exp(\frac{\pi \Ii}{6}), \exp(-\frac{\pi \Ii}{12})$ and $-\exp(-\frac{\pi \Ii}{12}).$ In particular, $\rho$ has polynomial-growth in this case.

\subsection{A consequence of polynomial-growth}
In Theorem~\ref{maintheorem} we achieved polynomial-growth of the Fourier coefficients, by showing a bound of the form $\norm{\X(\tau)}\ll y^{-k-2\alpha}$ when $\X(\tau)$ is a holomorphic vvaf. In this process, polynomial-growth of the associated representation $\rho$ played a crucial role. We obtained such a growth of $\rho$ assuming that, all the images of the parabolic elements have only unitary eigenvalues. However, even if we do not have this assumption, we could still consider a meromorphic function $\X:\mathbb{H}\to \C$ which satisfies the functional behavior. In other words, suppose that $\X|_k\gamma=\rho(\gamma)\X,~\forall \gamma\in \rG.$ In this more general situation, one may naturally ask whether we still have a polynomial-growth on $\X(\tau).$  To answer this question, we shall prove Theorem~\ref{thm:converse}
\begin{proof}[Proof of Theorem~\ref{thm:converse}]
Let us consider the space generated by the values given by $\X(\tau).$ More precisely, let us first consider $\mathbb{H}’$ to be $\mathbb{H} \setminus \{\text{poles~of}~\X(\tau)\}$ and $\mathcal{W}=\text{Span}_{\mathbb{C}}\left\{\X(\tau)\mid \tau \in \mathbb{H}’\right\}.$ Of course, here in $\mathcal{W},$ we are taking finite linear combination of $\X(\tau)$ over $\mathbb{C}.$ 

It is clear that $\mathcal{W}$ is a vector subspace of $\mathbb{C}^m,$ and hence having finite dimension. We now want to show that $\mathcal{W}$ is a representation of $\rG.$ For this, we use the functional equation $j(\gamma,\tau)^{-k}\X(\gamma \tau)=\rho(\gamma)\X(\tau),$ and note that $j(\gamma,\tau)\neq 0$ for any $\gamma \in \rG, \tau\in \mathbb{H}’.$ In particular, we have an action of $\rG$ on $\mathcal{W}$ given by $\rho.$ Therefore, $\mathcal{W}$ can be considered as a sub-representation of $\mathbb{C}^m.$ Let us first prove $(a)$. In this case $\rho$ is irreducible, and $\X(\tau)$ is a non-zero function, therefore $\mathcal{W}$ is isomorphic to $\mathbb{C}^m$. Let us now fix a basis $\left\{\X(\tau_1),\X(\tau_2),\cdots, \X(\tau_m)\right\}$ of $\mathcal{W}.$ We then have the following estimate for part $(a).$
\begin{align*}\norm{\rho(\gamma)}\ll \text{sup}\left\{\frac{\norm{\rho(\gamma)\X(\tau_i)}}{\norm{\X(\tau_i)}}\right\}_{1\leq i\leq m}&=\text{sup}\left\{\frac{|c\tau_i+d|^{-k}\norm{\X(\gamma \tau_i)}}{\norm{\X(\tau_i)}}\right\}_{1\leq i\leq m}\\
&\ll \text{sup}\left\{\frac{|c\tau_i+d|^{-k} |\text{im}(\gamma \tau_i)|^{-\zeta}}{\norm{\X(\tau_i)}}\right\}_{1\leq i\leq m}\\
&\ll \text{sup} \left\{\frac{|c\tau_i+d|^{2\zeta-k} |\text{im}(\tau_i)|^{-\zeta}}{\norm{\X(\tau_i)}}\right\}_{1\leq i\leq m}\\
&\ll |c^2+d^2|^{\zeta-k/2}\ll \norm{\gamma}^{2\zeta-k},
\end{align*}
where we are writing $\gamma=\tmt{a}{b}{c}{d}.$ On the other hand for part $(b),$ using the bound $\norm{X(x+iy)}\ll \max_{0\leq j\leq m-1}\{|x+iy|^{j}y^{-\zeta}\},$ we have the following estimate
\begin{align*}\norm{\rho(\gamma)}\ll \text{sup}\left\{\frac{\norm{\rho(\gamma)\X(\tau_i)}}{\norm{\X(\tau_i)}}\right\}_{1\leq i\leq m}&=\text{sup}\left\{\frac{|c\tau_i+d|^{-k}\norm{\X(\gamma \tau_i)}}{\norm{\X(\tau_i)}}\right\}_{1\leq i\leq m}\\
&\ll \text{sup}\left\{\frac{|c\tau_i+d|^{-k} |\gamma \tau_i|^{j}|\text{im}(\gamma \tau_i)|^{-\zeta}}{\norm{\X(\tau_i)}}\right\}_{\substack{0\leq j\leq m-1\\1\leq i\leq m}}\\
&\ll \text{sup} \left\{\frac{|c\tau_i+d|^{2\zeta-k} \norm{\gamma}^{j}}{\norm{\X(\tau_i)}}\right\}_{\substack{0\leq j\leq m-1\\1\leq i\leq m}}\\
&\ll \max\{\norm{\gamma}^{j+2\zeta-k}\}_{0\leq j\leq m-1},
\end{align*}

Now for part $(c)$ if $\rho$ is not necessarily irreducible, $\mathcal{W}$ may be a proper subspace of $\mathbb{C}^m.$ If this is whole $\mathbb{C}^m,$ then the previous argument gives the desired growth of $\rho.$ If $\mathcal{W}$ is a proper subspace of $\mathbb{C}^m$ of dimension $m',$ say. Then we can consider a basis of $\mathbb{C}^m,$ of the form $\left\{\X(\tau_1),\X(\tau_2),\cdots, \X(\tau_{m'}),v_{m'+1},\cdots v_m\right\}.$ Then the block of $\rho$ corresponding to $\left\{\X(\tau_1),\X(\tau_2),\cdots, \X(\tau_{m'})\right\}$ has a similar polynomial-growth, following the same argument as in the irreducible case.
\end{proof}

\subsection{Remarks}
\subsubsection{} Part $(a)$ of Theorem~\ref{thm:converse} could be considered as a converse statement to the admissible case of Theorem~\ref{maintheorem}, and part $(b)$ to the more general logarithmic case.

\subsubsection{} Theorem~\ref{thm:converse} is meaningful if there exists at least one case where the associated representation does not have polynomial-growth and the corresponding meromorphic function satisfies the functional property. We could then say, such a meromorphic function does not have polynomial-growth. In fact, given any representation it is possible to construct a meromorphic function which satisfies the functional-property. For instance, one could consider $\rG=\text{PSL}_2(\Z)$ and $\rho$ be the representation as in Example~\ref{ex:non-poly}. Fix a fundamental domain $\mr{F}_{\rG}$ of $\rG,$ and consider a meromorphic function $\X:\mr{F}_{\rG}\to \mathbb{C}^3$ which is $0$ on the boundary of $\widehat{\mr{F}_{\rG}},$ and then extend $\X$ to the whole upper half-plane by the functional property.

\section{$L$-functions and analytic continuations}\label{L-funct} 
Knopp and Mason remarked~\cite{Knopp4} that it is possible to attach an $L$-function to admissible vvmf for modular group with analytic continuation. In this section, we generalize this notion to both admissible and logarithmic vvaf for Fuchsian groups of the first kind. The proof of functional equation is classical, however we will briefly discuss the proof to remain self contained. In this connection, reader will find the article~\cite{KL2016} useful, where the authors discussed about the special values of $L$-function attached to vvmf for the modular group. In this section, we rather focus on the analytic continuation of these $L$-functions attached in general to  vector-valued automorphic forms for Fuchsian groups of the first kind.

\subsection{Admissible case} We first assume that $0,\infty$ are cusps of $\rG$ and $\X$ be an admissible cuspform. We start with defining
\[L(\X,\mathfrak{s})=\sum_{n\geq 0}\frac{\X_{[n]}}{(n+\Lambda)^{\mathfrak{s}}}\]
in $\text{Re}(\mathfrak{s})>k/2+\alpha+1,$ where $\alpha$ is the constant as in Theorem~\ref{maintheorem}, $\X_{[n]}$ is the $n^{\text{th}}$-Fourier coefficient of $\X$ at $\infty,$ and $(n+\Lambda)$ is the matrix 
$$\mathrm{diag}\left(n+\mu(\lambda_1),n+\mu(\lambda_2),\cdots,n+\mu(\lambda_m)\right),$$ 
where each $\lambda_i$ is an eigenvalue of $\rho(t_{\infty}).$ We also define the completed $L$-function $\widetilde{\Lambda}(\X,\mathfrak{s})$ to be $(2\pi)^{-\mathfrak{s}}\Gamma(\mathfrak{s})L(\X,\mathfrak{s}).$ Recalling that $\Gamma(\mathfrak{s})=\int_{0}^{\infty}e^{-y}y^{\mathfrak{s}-1}dy,$ we have
\begin{align*}
\int_{0}^{\infty}\X(ihy)y^{\mathfrak{s}-1}dy &=\int_{0}^{\infty}\sum_{n\geq 0} \X_{[n]}e^{-2\pi y(n+\Lambda)}y^{\mathfrak{s}-1}dy\\
&= (2\pi)^{-\mathfrak{s}}\Gamma(s)L(\X,\mathfrak{s})=\widetilde{\Lambda}(\X,\mathfrak{s}).\
\end{align*}
Theorem~\ref{maintheorem} implies that $\widetilde{\Lambda}(\X,\mathfrak{s})$ is analytic in the region $\mathrm{Re}(\mathfrak{s})> k/2+\alpha.$ This is because, for each component $\X_j$ of $\X$ we have
\begin{align*}
\Big|\int_{0}^{\infty}\X_j(ihy)y^{\mathfrak{s}-1}dy\Big| &\leq \int_{0}^{Y}\Big|\X_j(ihy)\Big|y^{\mathfrak{s}-1}dy+\int_{Y}^{\infty}\Big|\X_j(ihy)\Big|y^{\mathfrak{s}-1}dy\\
&\leq \int_{0}^{Y}|hy|^{-k/2-\alpha}y^{\mathfrak{s}-1}dy+ \int_{Y}^{\infty}\mathrm{exp}(-cy)y^{\mathfrak{s}-1}dy,
\end{align*}
where $c\, (>0)$ is coming from the exponential decay of $\X$ in a neighborhood of $\infty.$ In particular, the second integral is bounded for any $\mathfrak{s}.$ To bound the first integral it is enough to consider the range only from $0$ to $1$ and after performing a change of variable we are left with the integral $\int_{1}^{\infty}y^{k/2+\alpha-\mathfrak{s}-1}dy,$ which converges for $\mathrm{Re}(\mathfrak{s}) >k/2+\alpha.$

Consider $S=\tmt{0}{-1}{1}{0}$ and note that $\Y=\X|_KS,$ is a vvaf for the group $S^{-1}\rG S.$ We now have
\begin{align*}
\int_{0}^{\infty}\Y(ihy)y^{\mathfrak{s}-1}ds &=(hi)^{-k}\int_{0}^{\infty}y^{\mathfrak{s}-k-1}\X(i/hy)dy\\
&=(hi)^{-k}\int_{0}^{\infty}\sum_{n \geq 0}\X_{[n]}e^{-2\pi (n+\Lambda)/h^2y}y^{\mathfrak{s}-k-1}dy\\
&=-(hi)^{-k}\int_{0}^{\infty}\sum_{n \geq 0}\X_{[n]}e^{-2\pi y(n+\Lambda)/h^2}y^{k-\mathfrak{s}-1}dy\\
&=-(hi)^{-k}h^{2k-2s}\widetilde{\Lambda}(\X,k-\mathfrak{s}).
\end{align*}
Moreover, since $\infty$ is a cusp of $S^{-1}\rG S,$ arguing similarly as before, we get an analytic continuation for $\int_{0}^{\infty}\Y(ihy)y^{\mathfrak{s}-1}ds,$ to $\mathrm{Re}(\mathfrak{s})> k/2+\alpha.$ In particular, the completed $L$-function $\widetilde{\Lambda}(\X,\mathfrak{s})$ has analytic continuation to $\mathrm{Re}(\mathfrak{s})< k/2-\alpha$ as well.

\subsection{Logarithmic case} In this more general settings, we define the corresponding $L$-function as
\[L(\X,\mathfrak{s})=\sum_{\lambda \in \Sp}\sum_{j=0}^{m(\lambda)-1}\sum_{n\geq 0}\frac{\X_{[\lambda,j,n]}}{(n+\mu(\lambda))^{\mathfrak{s}+j}}\]
in $\text{Re}(\mathfrak{s})>k/2+\alpha+1,$ and the completed $L$-function
\[\widetilde{\Lambda}(\X,\mathfrak{s})=(2\pi)^{-\mathfrak{s}}\sum_{\lambda\in \Sp}\sum_{j=0}^{m(\lambda)-1}(-1)^{j}\Gamma(s+j)\sum_{n\geq 0}\frac{\X_{[\lambda,j,n]}}{(n+\mu(\lambda))^{\mathfrak{s}+j}}.\]
Following the arguments as in the previous case, we have
\begin{align*}
\int_{0}^{\infty}\X(ihy)y^{\mathfrak{s}-1}dy &=\int_{0}^{\infty}\sum_{\lambda\in \Sp}\sum_{j=0}^{m(\lambda)-1}(-2\pi)^{j}\sum_{n\geq 0} \X_{[\lambda,j,n]}e^{-2\pi y(n+\mu(\lambda))}y^{\mathfrak{s}+j-1}dy\\
&= (2\pi)^{-\mathfrak{s}}\sum_{\lambda\in \Sp}\sum_{j=0}^{m(\lambda)-1}(-1)^{j}\Gamma(\mathfrak{s}+j)\frac{\X_{[\lambda,j,n]}}{(n+\mu(\lambda))^{\mathfrak{s}+j}}\\
&=\widetilde{\Lambda}(\X,\mathfrak{s}),~\text{as~defined~earlier.}
\end{align*}
The logarithmic case of Theorem~\ref{maintheorem} shows that $L(\X,\mathfrak{s})$ converges absolutely in the region $\mathrm{Re}(\mathfrak{s})> k/2+\alpha.$ To show the convergence of $\int_{0}^{\infty}\X(ihy)y^{\mathfrak{s}-1}dy$ in $\text{Re}(\mathfrak{s})>k/2+\alpha$, we basically need to know the growth of $\X(ihy).$ From the proof of logarithmic case of Theorem~\ref{maintheorem} we know that $\norm{\X(iy)}\ll y^{-k/2-\alpha}$ for $y<1,$ and from the logarithmic expansion of $\X,$ we have exponential decay of $\X(iy)$ as $y\to \infty.$ We now obtain the desired convergence following similar arguments as in the admissible case. Now, considering $\Y=\X|_S,$ we have
\begin{align*}
& \int_{0}^{\infty}\Y(ihy)y^{\mathfrak{s}-1}ds =(hi)^{-k}\int_{0}^{\infty}y^{\mathfrak{s}-k-1}\X(i/hy)dy\\
&=(hi)^{-k}\int_{0}^{\infty}\sum_{\lambda\in \Sp}\sum_{j = 0}^{m(\lambda)-1}(-2\pi/h^2)^{j}\sum_{n\geq 0}\X_{[\lambda,j,n]}e^{-2\pi (n+\mu(\lambda))/yh^2}y^{\mathfrak{s}-k-j-1}dy\\
&=-(hi)^{-k}\int_{0}^{\infty}\sum_{\lambda\in \Sp}\sum_{j = 0}^{m(\lambda)-1}(-2\pi/h^2)^{j}\sum_{n\geq 0}\X_{[\lambda,j,n]}e^{-2\pi y(n+\mu(\lambda))/h^2}y^{k+j-\mathfrak{s}-1}dy\\
&=-(hi)^{-k}h^{2k-2\mathfrak{s}}\widetilde{\Lambda}(\X,k-\mathfrak{s}).
\end{align*}
In particular, the completed $L$-function $\widetilde{\Lambda}(\X,\mathfrak{s})$ has analytic continuation to $\text{Re}(\mathfrak{s})\leq k/2-\alpha.$

\begin{note}\rm$\;$ In both cases we are assuming that  $0$ and $\infty$ are cusps of $\rG.$ In general, since $\rG$ is a Fuchsian group of the first kind, $\rG$ has at least two distinct cusps and for each pair of cusps $\cu_1,\cu_2$ of $\rG,$ there exists some $\gamma \in \mathrm{PSL}_2(\mathbb{R})$ such that $\gamma \cu_1=\infty, \gamma \cu_2=0.$ For each such $\gamma,$ let us consider $L_{\cu_1,\cu_2,\gamma}(\X,\mathfrak{s})=L(\X|_k\gamma,\mathfrak{s}).$ From the construction it follows that $\gamma\rG\gamma^{-1}$ contains the cusps $0$ and $\infty.$ Now if $\gamma'$ is any other element sending $\cu_1$ to $\infty$ and $\cu_2$ to $0,$ then $\gamma'\gamma^{-1}$ fixes both $0$ and $\infty.$ Therefore, $\gamma'\gamma^{-1}$ must be of form $\tmt{a}{0}{0}{{1/a}},$ and hence  $\X|_k\gamma'=\X|_k\gamma|_k\tmt{a}{0}{0}{{1/a}}.$ In particular, the Fourier coefficients of $\X|_k\gamma$ and $\X|_k\gamma'$ differs by a scalar factor. We can now fix $\X_{\cu_1,\cu_2}$ to be a vvaf in this family having the first non zero Fourier coefficient of norm $1.$ With this set up we can now finally define $L_{\cu_1,\cu_2}(\X,\mathfrak{s})=L(\X_{\cu_1,\cu_2},\mathfrak{s}),$ and this gives us a (non-empty) family of $L$-functions indexed by pairwise distinct cusps.
\end{note}

\section{Exponential sums and growth on average}\label{sec:app} 
Studying exponential sums associated to arithmetic functions is of great interest in number theory.  When $f$ is a cusp form of weight $k$ and level $N$, by the standard bound on the Fourier coefficients, one can show that
\begin{equation}\label{eq:Sf}S_f(\alpha,X)=\sum_{1\leq n\leq X} f_{[n]} e(n\theta)\ll X^{k/2}\log X,\end{equation}
where we make use of the standard notation $e(z):=e^{2i\pi z}$ and stick to it throughout the section.
The extra $\log$ factor in~\eqr{Sf} was later removed by Jutila~\cite{Jutila}. In this section, we shall first consider the analogous exponential sums for holomorphic vvaf and show that  how our growth result gives a bound of order $X^{\sigma(k/2+\alpha)}\log X.$ Here $\sigma=2$ for the holomorphic forms and $1$ for the cusp forms. In this section, our aim is to study the analogous exponential sums associated to Fourier coefficients of holomorphic vvaf associated to Fuchsian groups of the first kind.  
\subsection{Admissible case} Since we are in the admissible case, write 
\[ \X(z)=\Big(\sum_{n=0}^{\infty} \X_{[i,n]}\tq^{n+\mu(\lambda_i)}\Big)_{0 \leq i\leq m-1}.\]
Let us first consider the exponential sums associated to the components of the Fourier coefficients as
\[S_i(\X,\theta,X)=\sum_{0\leq n< X} \X_{[i,n]}e(n\theta),~0\leq i\leq m-1.\]
For any $y>0,$ we have $\X_{[i,n]}=\frac 1 h\int_{0}^{h}\X_i(\tau)e\left(-\frac{\tau}{h}(n+\mu(\lambda_i))\right)~dx.$ Therefore
\[S_i(\X,\theta,X)=\frac 1 h\int_{0}^{h}\X_i(\tau)e\left(-\frac{\tau}{h}\lambda_i\right)\sum_{0\leq n< X} e\left(n\left(-\frac{\tau}{h}+\theta\right)\right)~dx.\]
The sum in the right hand side is a geometric progression, and this gives us
\[\left|\sum_{0\leq n< X} e\left(n\left(-\frac{\tau}{h}+\theta\right)\right)\right|\ll \left|\frac{1-e\left(X\left(-\frac{\tau}{h}+\theta\right)\right)}{1-e\left(-\frac{\tau}{h}+\theta\right)}\right|.\] 

Note that $|e\left(X\left(-\frac{\tau}{h}+\theta\right)\right)|=e^{2\pi Xy/h}$ and also from Theorem~\ref{maintheorem} we have $|\X_i(\tau)|\ll y^{-\sigma(k/2+\alpha)}$. In particular,
$$|S_i(\X,\theta,X)|\ll y^{-\sigma(k/2+\alpha)}e^{2\pi Xy/h}\frac 1 h\int_0^h\frac{1}{\left|1-e\left(-\frac{\tau}{h}+\theta\right)\right|} ~dx.$$
Replacing $x$ by $x+h\theta$ and using the periodicity of $e\left(-\frac\tau h\right)$,
\begin{align*}
   \int_{0}^{h} \frac{1}{\left|1-e\left(-\frac\tau h+\theta\right)\right|}~dx &= \int_{-h\theta}^{h-h\theta}\frac{1}{\left|1-e\left(-\frac\tau h\right)\right|}~dx
    =\int_{-h/2}^{h/2}\frac{1}{\left|1-e\left(-\frac\tau h\right)\right|}~dx\\
   & \ll\int_{0}^{h/2}\frac{1}{|\frac{\tau}{h}|}~dx
    = h\left(\int_0^y\frac{1}{|\tau|}~dx+\int_y^{h/2}\frac{1}{|\tau|}~dx\right)\\
   & \ll h\left(1 + \log \frac{h}{y}\right).
\end{align*}
Thus,
$$|S_i(\X,\theta,X)| \ll  y^{-\sigma(k/2+\alpha)} e^{2\pi Xy/h}\left(1+ \log\frac{h}{y}\right).$$
One now gets $|S_i(\X,\theta,X)|\ll X^{\sigma(k/2+\alpha)} \log X$ by taking $y=h/X.$

Doing a little more delicate analysis we can obtain a better bound of the Fourier coefficients on average. To be more precise, we shall now give a stronger bound on $\sum_{n\leq X}\norm{\sum \X_{[n]}}^2.$ Before that, let us first discuss the known results for scalar case. When $f$ is a (scalar-valued) cusp form of weight $k$ (and of level $N,$ say) then Rankin~\cite[Theorem 1]{Rankin39} showed that
\[\sum_{1\leq n\leq X}|f_{[n]}|^2=cX^{k}+O(x^{k-2/5}),\]
where $c>0$ is a computable constant. Writing $z=x+iy,$ we have
\begin{align*}
\sum_{0\leq n\leq X} \norm{\X_{[n]}}^2e^{-4\pi ny} &\leq \sum_{0\leq i\leq m-1}  \int_{0}^{h}|\X_i(x+iy)|^2~dx\\
& = \int_{0}^{h}\norm{\X(x+iy)}^2~dx \\
& \ll y^{-2k-4\alpha},
\end{align*}
for any $y>0.$ So in particular, taking $y=\frac{1}{X}$ we obtain for holomorphic forms that $\sum_{0 \leq n\leq X}\norm{\X_{[n]}}^2\ll X^{2k+4\alpha}.$ Similarly for cusp forms we get
\begin{equation}\label{eq:avg2}
    \sum_{0 \leq n\leq X}\norm{\X_{[n]}}^2\ll X^{k+2\alpha},
\end{equation}

\subsection{Logarithmic case} Similarly, we can now consider the exponential sum 
\[S(\X,\theta,X)=\sum_{\lambda\in \Sp}\sum_{j=0}^{m(\lambda)-1}\sum_{0\leq n\leq X}\X_{[\lambda,j,n]}e(n\theta).\]
We can recover $\X_{[\lambda,j,n]}$'s by integration as in the previous case. This again brings us to studying the integrals of form
\[S_{\lambda,j,i}(\X,\theta,X)=\frac{1}{h}\int_{0}^{h}h_{i,j,\lambda}(\tau)e\left(-\frac{\tau}{h}\mu(\lambda)\right)\sum_{0\leq n\leq X}e\left(n\left(-\frac{\tau}{h}+\theta\right)\right)~dx,\] where $h_{i,j,\lambda}$’s are defined in the proof of Lemma~\ref{lem:LC3}. Now to bound this exponential sum, we only need to bound $|h_{i,j,\lambda}(\tau)|$'s. This we can do by~(\ref{eqn:lvaf bound}) in the proof of Theorem~\ref{maintheorem}, and then following the same argument as in the previous case, we have
\begin{equation}\label{eqn:exp}
|S(\X,\theta,X)| \ll X^{\sigma(k/2+\alpha)} \log X,
\end{equation}
where $\sigma=2$ for the holomorphic case and $1$ for the cusp forms. Note that~(\ref{eqn:exp}) also gives us a bound on sum of squares of Fourier coefficients. However, we should aim for a bound like~(\ref{eq:avg2}). That is also possible in this case, because $h_{i,j,\lambda}$'s are $\widetilde{q}$-expansions and we know their growth, due to the logarithmic case of Theorem~\ref{maintheorem}.  


\section*{Acknowledgements}
The authors would like to thank the Georg-August Universit\"at G\"ottingen,  Technische Universit\"at Dresden, Max Planck Institue f\"ur Mathematics, Bonn, and IMPA (Instituto Nacional de Matemática Pura e Aplicada), where much of the work on this article was accomplished, for their hospitality. 

The authors wish to thank Harald Helfgott and Valentin Blomer for their support and encouragement at the initial stage of this work, and would like to extend their thanks to Samuel Patterson for several discussions on the subject during the writing of this article.

The work of the first author is supported by ERC consolidator grants  648329 and 681207, and the second author is supported by ERC consolidator grant 648329.


\nocite{}
 \bibliographystyle{abbrv}
 \bibliography{vvaf}
\end{document}